\documentclass[a4paper,12pt]{article}
\usepackage[utf8]{inputenc}
\usepackage[T1]{fontenc}
\usepackage[french,english]{babel}
\usepackage{amsfonts,amssymb,amsmath,amsthm}
\usepackage[margin=2.8cm]{geometry}
\usepackage{graphicx}
\usepackage{dirtytalk}

\usepackage{tikz}

\usepackage{caption}

\usepackage{hyperref}

\hypersetup{colorlinks=false,hidelinks}

\date{}

\newcommand{\suchthat}{|\;}
\newcommand{\A}{\mathcal{A}}

\newcommand{\R}{\mathbb{R}}
\newcommand{\Z}{\mathbb{Z}}
\newcommand{\Q}{\mathbb{Q}}
\newcommand{\N}{\mathbb{N}}

\newcommand{\U}{\mathbb{U}}
\newcommand{\vst}{v_{\mbox{\scriptsize st}}}
\newcommand{\Dc}{\mathcal{D}}
\newcommand{\Ac}{\mathcal{A}}
\newcommand{\Xc}{\mathcal{X}}
\newcommand{\connect}{\#}
\newcommand{\order}{\mathcal{O}}
\newcommand{\lk}{\mbox{lk}}
\newcommand{\slk}{\mbox{slk}}

\newcommand{\chiS}{\chi_{\interp}}
\newcommand{\kiS}[2]{\chiS^{\scriptscriptstyle [#1,#2]}}
\newcommand{\interp}{
\begin{tikzpicture}[scale=0.2]
\useasboundingbox (-0.6,-1.2) rectangle (0.6,-0.4);
\begin{scope}[yshift=-1.2 cm]
\draw (-0.5,-0.5) .. controls (0,-0.5) and (0,0.5) .. (0.5,0.5);
\end{scope}
\end{tikzpicture}
}
\newcommand{\kiSneg}[2]{\chiSneg^{\scriptscriptstyle [#1,#2]}}
\newcommand{\chiSneg}{\chi_{\interpneg}}
\newcommand{\interpneg}{
\begin{tikzpicture}[scale=0.2]
\useasboundingbox (-0.6,-1.2) rectangle (0.6,-0.4);

\begin{scope}[yshift=- 1.2 cm]
\draw (-0.5,0.5) .. controls (0,0.5) and (0,-0.5) .. (0.5,-0.5);
\end{scope}

\end{tikzpicture}
}

\theoremstyle{plain}
\newtheorem{thma}{Theorem}[section]
\newtheorem{lema}[thma]{Lemma}
\newtheorem{coroa}[thma]{Corollary}
\newtheorem{propa}[thma]{Proposition}

\newtheorem{faita}[thma]{Fact}

\newtheorem{questions}[thma]{Questions}

\theoremstyle{definition}
\newtheorem{defia}[thma]{Definition}

\newtheorem{notaa}[thma]{Notation}
\newtheorem{rquea}[thma]{Remark}

\newcommand{\tata}{ \begin{tikzpicture} \useasboundingbox (-.3,-.1) rectangle (.3,.2);
\draw (-.2,0) -- (.2,0) (0,0) circle (.2);
\fill (-.2,0) circle (1.5pt) (.2,0) circle (1.5pt);
\end{tikzpicture}} 

\newcommand{\antisymbis}{\begin{tikzpicture}[scale=0.5]
\useasboundingbox (-1,-1) rectangle (4,1);
\coordinate (a) at (0,0);
\coordinate (b) at (30:1);
\coordinate (c) at (150:1);
\coordinate (d) at (0,-1);
\draw[dashed] (a) circle (1);
\draw [very thick] (d) -- (a) -- (c);
\draw [very thick] (a) --(b);
\draw [very thick] (a) node{$\bullet$};

\draw  [very thick] (1.5,-0.2)--(1.5,0.2);
\draw [very thick] (1.3,0)--(1.7,0);

\begin{scope}[xshift=3cm]
\coordinate (e) at (0,0);

\coordinate (f) at (30:1);
\coordinate (g) at (150:1);
\coordinate (h) at (0,-1);
\draw[dashed] (e) circle (1);
\draw [very thick] (e).. controls (1,0).. (g);
\draw [very thick] (e).. controls (-1,0).. (f);
\draw [very thick] (e) --(h);
\draw [very thick] (e) node{$\bullet$};

\draw [very thick] (e) node{$\bullet$};
\end{scope}
\end{tikzpicture}}

\newcommand{\Jacobibis}{\begin{tikzpicture}[scale=0.5]
\useasboundingbox (-2,-1) rectangle (8,1);

\draw[dashed] (0,0) circle (1);

\coordinate (a) at (0,0);
\coordinate (b) at (0,-0.5);
\coordinate (c) at (60:1);
\coordinate (d) at (120:1);
\coordinate (e) at (0,-1);
\coordinate (f) at (220:1);

\draw [very thick] (e)-- (a)--(c) ;
\draw [very thick] (a) -- (d);
\draw [very thick] (f) -- (b);
\draw (a) node{$\bullet$};
\draw (b) node{$\bullet$};

\draw  [very thick] (1.5,-0.2)--(1.5,0.2);
\draw [very thick] (1.3,0)--(1.7,0);

\begin{scope}[xshift=3cm]
\draw[dashed] (0,0) circle (1);

\coordinate (a) at (0,0);
\coordinate (b) at (60:0.5);
\coordinate (c) at (60:1);
\coordinate (d) at (120:1);
\coordinate (e) at (0,-1);
\coordinate (f) at (220:1);

\draw [very thick] (e) -- (a)--(c) ;
\draw [very thick] (a) -- (d);
\draw [very thick] (f) .. controls (1,0.25) .. (b);

\draw (a) node{$\bullet$};
\draw (b) node{$\bullet$};

\draw  [very thick] (1.5,-0.2)--(1.5,0.2);
\draw [very thick] (1.3,0)--(1.7,0);
\end{scope}
\begin{scope}[xshift=6cm]
\draw[dashed] (0,0) circle (1);

\coordinate (a) at (0,0);
\coordinate (b) at (120:0.5);
\coordinate (c) at (60:1);
\coordinate (d) at (120:1);
\coordinate (e) at (0,-1);
\coordinate (f) at (220:1);

\draw [very thick] (e) -- (a)--(c) ;
\draw [very thick] (a) -- (d);
\draw [very thick] (f) .. controls (1.4,0.25) .. (b);

\draw (a) node{$\bullet$};
\draw (b) node{$\bullet$};

\end{scope}

\end{tikzpicture}}

\title{On perturbative invariants of combed three-manifolds}
\author{Yohan Mandin-\/-Hublé}

\begin{document}

\maketitle

\abstract{We give a new definition of a universal finite type invariant of three-dimensional oriented rational homology spheres which counts configurations of trivalent graphs in such manifolds. 
Kontsevich introduced this invariant following Witten's study of the perturbative expansion of the Chern-Simons theory, using parallelizations of three-manifolds. In this article, we use combings instead of parallelizations to get a more flexible and convenient definition.}

\tableofcontents 

\section{Introduction}

We give a new definition of a universal finite type invariant $Z$ of three-dimensional oriented rational homology spheres which counts configurations of trivalent graphs in such manifolds. The invariant $Z$ is valued in a quotient of the real vector space formally generated by automorphism classes of trivalent graphs. This space is graded by half the number of vertices of the graphs. Following Witten's study of the perturbative expansion of the Chern-Simons theory in \cite{Witten}, Kontsevich introduced this invariant in \cite{ko} using parallelizations of three-manifolds. Kuperberg and Thurston proved its universality among invariants of integer homology spheres in \cite{kt}. Lescop proved its universality among invariants of rational homology spheres and gave more flexible definitions of $Z$ in~\cite{lesbook}. In this article, we give a new definition of $Z$ using combings instead of parallelizations.

  The main ingredients in the definition of $Z$ are special closed differential $2$-forms called \emph{propagators} on the configuration space of injections of two points in the manifold. Fixing the behaviour of propagators on the boundary of an appropriate compactification of this configuration space using a parallelization $\tau$ of the manifold $M$ outside a point gives rise to a topological invariant $Z(M,\tau)$ of $(M,\tau)$. The invariant $Z(M)$ is then defined as an explicit function of $Z(M,\tau)$, a constant $\beta$ called the \emph{beta anomaly}, and a $\Z$-valued invariant $p_1(\tau)$ of $\tau$ induced by the first Pontrjagin class. The map $p_1$ or some of its variants have been studied by Hirzebruch in \cite[Section 3.1]{hirzebruchEM}, by Kirby and Melvin in \cite{km}, and by Lescop in \cite[Chapter 5]{lesbook}.
  
  In this article, we show how to replace parallelizations by combings, which are nowhere vanishing vector fields on the manifold $M$ outside of a point. We define an invariant $Z(M,X)$ of a rational homology sphere $M$ equipped with a combing $X$.
We express $Z(M,X)$ as an explicit function of $Z(M)$, the above beta anomaly, and a $\Q$-valued invariant $p_1(X)$ of combings. The invariant $p_1$ was defined and studied by Lescop in \cite{lescomb}. It coincides with the invariant $p_1(\tau)$ of a parallelization $\tau$ when $X$ is the first vector of $\tau$. It generalizes an invariant of combings in closed $3$-manifolds defined and studied by Gompf in \cite{Go}. Our results give a more practical and flexible definition of the invariant $Z$ of a rational homology sphere.

 Our main motivation is the concrete computation of $Z$ from a Heegaard splitting of the manifold. Combings arise naturally in this setting. In \cite{lesHC},
 Lescop constructed a Morse propagator on a rational homology sphere equipped with a Morse function and an associated combing. She computed an invariant $\Theta(M,X)$ of a rational homology sphere $M$ equipped with a combing $X$ using such a Morse propagator. This invariant is equivalent to the degree one part of $Z(M,X)$. The new definition of $Z$ given in this article is easier to apply with Morse propagators.

\subsection{Specifying the setting}

In this article, a \emph{rational homology sphere} is a smooth, closed, oriented three-dimensional manifold with the same rational homology as $S^3$. Let $M$ be a rational homology sphere. Let $B(1)$ be the unit open ball in $\R^3$. We equip $M$ with a point $\infty\in M$ and an orientation-preserving diffeomorphism  from a neighbourhood $V_\infty$ of $\infty$ in $M$ to a neighbourhood $S^3 \setminus B(1)$ of $\infty$ in $S^3=\R^3 \cup \{\infty\}$. The manifold $M\setminus\{\infty\}$ is denoted by $\check{M}$. The tangent bundle of $M$ is denoted by $TM$. The unit tangent bundle to $M$ is $UM=(TM\setminus \{0\}) /\R_+^*$, where $\R_+^*$ acts by scalar multiplication. Let $\check{V}_\infty$ denote $V_\infty\setminus\{\infty\}$.

In this article, $X$ denotes a section of $U\check{M}$ that coincides with $(0,0,1)$ on $\check{V}_\infty$ via the identification provided by the above diffeomorphism. This identification will always be implicitly used in the rest of this article. We only consider homotopies of such sections that satisfy this condition at any time. A \emph{combing} is a homotopy class $[X]$ of such a section $X$.

Any three-dimensional oriented smooth manifold is parallelizable. In this article, parallelizations $\tau\colon \check{M}\times \R^3 \to T\check{M}$ of $\check{M}$ are assumed to coincide with the standard parallelization of $\R^3$ on $\check{V}_\infty$. We only consider homotopies of such parallelizations that satisfy this condition at any time.

The disjoint union of graphs equips the target vector space of $Z$ with the structure of an algebra. There exists an invariant $z$ of rational homology spheres such that $Z = \exp (z)$. The invariant $z$ counts embeddings of connected trivalent graphs in $M$. We denote the degree $n$ part of $z$ by $z_n$.

In this article, we define an invariant $z(M,[X])$ of a rational homology sphere $M$ equipped with a combing $[X]$. This is the content of Theorem~\ref{thm_inv_combing}.
We prove the variation formula $z(M,[X])=z(M)+1/4 p_1(X)\beta$. This is Theorem~\ref{thm_inv_var}. These theorems are the two main results of this paper.

The introduction of this article is organised as follows. In Section~\ref{sec_comb_multi_intro}, we introduce multisections of the bundle $X^\perp=T\check{M}/\R X$ for a combing $[X]$ on $\check{M}$. In Section~\ref{sec_propa_comb_intro}, we introduce propagating forms associated with a combing $[X]$ and a multisection of $X^\perp$. In Section~\ref{sec_graphs}, we introduce the necessary terminology on trivalent graphs. In Section~\ref{sec_main_result}, we state Theorems \ref{thm_inv_combing} and \ref{thm_inv_var}. In Section~\ref{sec_dual_main_result}, we define propagating chains associated with a combing $[X]$ and a multisection of $X^\perp$, and we state a dual version of Theorem~\ref{thm_inv_combing} in Theorem~\ref{thm_inv_combing_dual}.

\subsection{Combings and multisections}
\label{sec_comb_multi_intro}

In this article, we only consider Riemannian metrics $g$ on $\check{M}$ that coincide with the standard Euclidean metric on $V_\infty \setminus \{\infty\}$. Let $g$ be such a metric. We identify $U\check{M}$ with $\{v\in T\check{M} \suchthat g(v,v)=1\}\subset T\check{M}$.
 
	 Let $[X]$ be a combing of $\check{M}$. We identify $X^\perp$ with the bundle $X^\perp\subset T\check{M}$ of vectors orthogonal to $X$. Let $UX^\perp \subset T\check{M}$ be the intersection $U\check{M}\cap X^\perp$. The Lie group $S^1$ acts freely and transitively on $UX^\perp$ by rotations of axis $X$. This action equips $UX^\perp$ with the structure of a principal $S^1$-bundle. 
	 
	The \emph{Euler class} $e(UX^\perp)\in H^2(\check{M};\Z)\approx  H^2(M;\Z)$ of the principal $S^1$-bundle $UX^\perp$ represents the obstruction to the existence of a nowhere vanishing section of $X^\perp$. It is precisely defined in Section~\ref{sec_multisection}. Let $\order(X)$ denote the order of $e(UX^\perp)$ in the finite group $H^2(\check{M};\Z)$.
	 
	  Let $p\in \N\setminus\{0\}$. Let $\mathbb{U}_p$ denote the subgroup of $S^1$ of $p$-th roots of unity. Consider the free action of $\U_p\subset S^1$ on $UX^\perp$. As we will quickly show in Section~\ref{sec_multisection}, when $p$ is a multiple of $\order(X)$, the quotient bundle $UX^\perp_{p}=UX^\perp / \U_p$ is trivial. Let $f_p\colon UX^\perp \to UX^\perp_{p}$ denote the quotient map.
	  
	  \begin{defia}\label{def_multisection}
Let $[X]$  be a combing on $\check{M}$.
 Let $\vst\colon V_\infty \to  UX^\perp|_{\check{V}_\infty}$ be the constant section $(1,0,0)$. Let $p\in\N\setminus\{0\}$ be a multiple of $\order(X)$. A \emph{$p$-multisection of $UX^\perp$} is a section of $UX^\perp_p$ that coincides with $f_p\circ \vst$ on $V_\infty\setminus\{\infty\}$. A \emph{multisection of $UX^\perp$} is a pair $(p,\nu^p_X)$ where $p$ is as above and $\nu^p_X$ is a $p$-multisection of $UX^\perp$.
\end{defia}

A $p$-multisection of $UX^\perp$ locally lifts as in the following definition.
\begin{defia}\label{def_proj_combing}
Let $[X],p,\nu_X^p$ be as in Definition \ref{def_multisection}. Let $B$ be an open ball embedded in $\check{M}$. A \emph{lift $\nu_X\colon B \to UX^\perp|_{B}$ of $\nu^p_X$ over $B$} is a section of $UX^\perp$ on $B$ such that $f_p\circ \nu_X= \nu^p_X$. Complete $(X,\nu_X)$ to an orthonormal basis $(\nu_X,w_X,X)$
of $TM|_{B}$, which induces the orientation of $M$. Let $p_{(\nu_X,w_X,X)} \colon UM|_{V}\to S^2$ be the projection associated to this basis that sends $\nu_X$ to $(1,0,0)$, $w_X$ to $(0,1,0)$, and $X$ to $(0,0,1)$.
\end{defia}

\subsection{Propagating forms associated with multisected combings}	 
\label{sec_propa_comb_intro}
	
The space $\check{C}_2(\check{M})$ of injections of two points in $\check{M}$ has the same rational homology as $S^2$. It admits a natural Fulton-Mac Pherson type \cite{FultonMcP} compactification $C_2(M)$, which is a smooth $6$-dimensional smooth manifold with corners. Throughout this article, our general reference for configuration spaces of points in $M$ and their compactifications is \cite[Chapter 8]{lesbook}.
 Let $||\cdot||$ denote the standard Euclidean norm on $\R^3$.  The Gauss map $G_{S^3}$ is defined on $\check{C}_2(\R^3)$ to be
$$G_{S^3}(x,y)= \frac{y-x}{||y-x||}.$$
It extends to a smooth map $G_{S^3}$ on $C_2(S^3)$. 

There is a natural inclusion of $U\check{M}$ into $\partial C_2(M)$.\footnote{We include $U\check{M}$ into $\partial C_2(M)$ as follows. Let $(m,v)$ be in $U\check{M}$. Let $\gamma \colon ]-1,1[ \to \check{M}$ be an immersion such that $\gamma(0)=m$ and the direction of $\gamma'(0)$ is $v$. The function $((\gamma(0),\gamma(t)) \in \check{C}_2(\check{M}))_{t\in ]0,1[}$ admits a limit $c \in \partial C_2(M)$ at $t=0$. We identify $c$ and $(m,v)$.} There is a smooth \emph{Gauss map} $G_M\colon (\partial C_2(M) \setminus U\check{M}) \cup U\check{M}|_{\check{V}_\infty} \to S^2$. It is defined using $G_{S^3}$ and our chosen identification between $V_\infty$ and $S^3\setminus B(1)$. \emph{Propagating forms of $C_2(M)$}  are closed differential forms of degree $2$ on $C_2(M)$ that restrict to
$\partial C_2(M) \setminus U\check{M}$ as the pull-back of a $2$-form of volume one on $S^2$ by the Gauss map $G_M$.

A parallelization $\tau$ of $\check{M}$ induces a smooth map $p_{\tau}\colon \partial C_2(M) \to S^2$ that coincides
with $G_M$ on $(\partial C_2(M) \setminus U\check{M}) \cup U\check{M}|_{\check{V}_\infty}$. It is defined on $U\check{M}$ by the relation $p_{\tau}([\tau(m,v)])=v/||v||$, for any $(m,v)\in \check{M} \times (\R^3\setminus\{0\})$. A \emph{propagating form of $(C_2(M),\tau)$} is a propagating form of $C_2(M)$ that restricts to $\partial C_2(M)$ as the pull-back of a volume-one form on $S^2$ by the map $p_{\tau}$.

For $\theta \in S^1$, let $R_{\theta}\colon S^2 \to S^2$ denote the rotation about $(0,0,1)$ by an angle $\theta$. Consider the action of $S^1$ on $S^2$ given by $\theta\mapsto R_\theta$. Let $p\in\N\setminus\{0\}$. We say that a differential form $\omega$ on $S^2$ is $\U_p$-invariant if it is invariant with respect to the action of $\U_p\subset S^1$, that is
$$\forall k\in\{1,\dots,p\}, R_{\frac{2k\pi}{p}}^*\omega=\omega.$$

Let $[X],p,\nu_X^p,B$ be as in Definition \ref{def_proj_combing}. Let $\nu_X,\nu_X^\prime$ be two lifts of $\nu^p_X$ over $B$. There exists $k\in\mathbb{N}$ such that $$ p_{(\nu_X^\prime,w_X^\prime,X)}= R_{\frac{2k\pi}{p}}\circ  p_{(\nu_X,w_X,X)}.$$
Let $\omega$ be a $\U_p$-invariant form on $S^2$.  We have $p_{(\nu_X^\prime,w_X^\prime,X)}^*(\omega)=p_{(\nu_X,w_X,X)}^*(\omega)$.  There exists $k'\in\mathbb{N}$ such that on $U\check{M}|_{B\cap V_\infty}$, we have $p_{(\nu_X,w_X,X)}=R_{2k'\pi/p}\circ G_M$. So we have $p_{(\nu_X,w_X,X)}^*(\omega)=G_M^*(\omega)$ on $U\check{M}|_{B\cap V_\infty}$. This shows that the following definition is consistent.

\begin{defia}\label{def_operator}
Let $[X],p, \nu^p_X $ be as in Definition \ref{def_multisection}. We define a map $p_{(X,\nu_X^p)}^*$ from the space of $\U_p$-invariant forms on $S^2$ to the space of forms on $\partial C_2(M)$ as follows. Let $\omega$ be an $\U_p$-invariant form on $S^2$. For any ball $B\subset \check{M}$ and any lift $\nu_X$ of $\nu_X^p$ on $B$, we define $p_{(X,\nu_X^p)}^*(\omega)$ on $U\check{M}|_{B}$  as $$p_{(X,\nu_X^p)}^*(\omega)=p_{(\nu_X,w_X,X)}^*(\omega).$$
On $\partial C_2(M)\setminus U\check{M}$, we set $p_{(X,\nu_X^p)}^*(\omega)=G_M^*(\omega)$.
\end{defia}

\begin{defia}\label{def_propa_combing}
Let $[X],p, \nu^p_X $ be as in Definition \ref{def_multisection}. A \emph{propagating form of $(C_2(M),g,X, \nu^p_X)$} is a closed $2$-form $\Omega$ on $C_2(M)$ such that $\Omega=p_{(X,\nu_X^p)}^*(\omega)$
 on $\partial C_2(M)$ for some $\U_p$-invariant form $\omega$ on $S^2$ of volume one. 
\end{defia}
Propagating forms of $(C_2(M),g,X, \nu^p_X)$ exist because of homological reasons. Indeed, the vector space $H^3(C_2(M),\partial C_2(M);\R)$ is trivial, so the inclusion $\partial C_2(M) \subset C_2(M)$ induces a surjective map from $H^2(C_2(M);\R)$ to $H^2(\partial C_2(M);\R)$.

\subsection{Spaces of trivalent graphs}
\label{sec_graphs}

A \emph{trivalent graph} $\Gamma$ is a set of half-edges $H(\Gamma)$ equipped with two partitions. One partition is the vertex partition. It is denoted by $V(\Gamma)$. Its elements are of cardinality $3$ and are called vertices. The other partition is the edge partition. It is denoted by $E(\Gamma)$. Its elements are of cardinality $2$ and are called edges. The \emph{degree} of a trivalent graph is half the number of its vertices.

An \emph{edge-orientation} of a trivalent graph $\Gamma$ is the data of a total order on each edge of $\Gamma$. A \emph{vertex-orientation} of a trivalent graph $\Gamma$ is the data of a total order up to cyclic permutation on each vertex of $\Gamma$.

An \emph{isomorphism of trivalent graphs} is a bijection between the underlying sets of half-edges that preserves the vertex partition and the edge partition. An \emph{isomorphism of trivalent graphs equipped with an additional structure} is moreover required to preserve this structure.

Let $n\geq 1$ be an integer. Let $\xi_n=1/((3n)!2^{3n})$. Let $A_n^c(\emptyset)$ be the real vector space freely generated by isomorphism classes of connected trivalent graphs of degree $n$ equipped with a vertex-orientation. Let $\A_n^c(\emptyset)$ denote the quotient of $A_n^c(\emptyset)$ by antisymmetry and Jacobi relations pictured in Equation \ref{eq_relation}. The definition of these relations can be found in \cite[Section 6.3]{lesbook}.

\begin{equation}\label{eq_relation}
\left\{
\begin{array}{c}
\antisymbis \\
\Jacobibis
\end{array}
\right.
.\end{equation}

Let $\Dc_n^c(\emptyset)$ be the set of isomorphism classes of pairs $(\Gamma,j)$ where $\Gamma$ is an edge-oriented  connected trivalent graph with $2n$ vertices without looped edges, and $j\colon E(\Gamma) \to \{1,\dots,3n\}$ is a bijection.

\subsection{Main result}
\label{sec_main_result}

Let $\Gamma$ be an edge-oriented connected trivalent graph without looped edges. The space of injections of $V(\Gamma)$ into $\check{M}$ is denoted by $\check{C}_{V(\Gamma)}(\check{M})$. A vertex-orientation $o(\Gamma)$ of $\Gamma$ induces an orientation of $\check{C}_{V(\Gamma)}(\check{M})$, which we implicitly use throughout the article. See \cite[Corollary 7.2]{lesbook}. Let $[\Gamma,o(\Gamma)]$ denote
the class of $(\Gamma,o(\Gamma))$ in $\A_n^c(\emptyset)$.

Let $e$ be an edge of $\Gamma$ going from a vertex $u_1$ to a vertex $u_2$ of $\Gamma$. Let $p_e\colon \check{C}_{V(\Gamma)}(\check{M}) \to \check{C}_2(\check{M})$ be the smooth map defined as $p_e\colon c \mapsto (c(u_1),c(u_2))$.

Let $(\Gamma,j)$ be in $\Dc_n^c(\emptyset)$. Let $(\omega_i)_{i\in\{1,\dots,3n\}}$ be a family of $3n$ propagating forms of $(C_2(M),g,X, \nu^p_X)$.  The element 
$\bigl(\int_{\check{C}_{V(\Gamma)}(\check{M})}\bigwedge_{e\in E(\Gamma)} p_e^*(\omega_{j(e)}) \bigr)[\Gamma,o(\Gamma)]$
of $\A_n^c(\emptyset)$ does not depend on $o(\Gamma)$. Let $I(M,\Gamma,j,(\omega_i)_{i\in\{1,\dots,3n\}})[\Gamma]$ denote this element.

\begin{thma}\label{thm_inv_combing}
Let $g$ be a Riemannian metric on $\check{M}$. Let $[X]$ be a combing on $\check{M}$. Let $p\in\N\setminus\{0\}$ be a multiple of $\order(X)$. Let $\nu^p_{X}$ be a $p$-multisection of $UX^\perp$. Let $(\omega_i)_{i\in\{1,\dots,3n\}}$ be a family of $3n$ propagating forms of $(C_2(M),g,X, \nu^p_{X})$. The element
 $$z_n(M,(\omega_i)_{i\in\{1,\dots,3n\}})=\xi_n  \sum_{(\Gamma,j)\in \Dc_n^c(\emptyset)}  I(M,\Gamma,j,(\omega_i)_{i\in\{1,\dots,3n\}})[\Gamma]$$
 of $\Ac_n^c(\emptyset)$ does not depend on the choice of the family $(\omega_i)_{i\in\{1,\dots,3n\}}$ of propagating forms of $(C_2(M),g,X,\nu^p_{X})$. It does not depend on the choices of $g$, $\nu^p_{X}$, and $p$. It only depends on $[X]$ and on the orientation-preserving diffeomorphism type of $M$. We denote this element by $z_n(M,X)$.
\end{thma}

We prove Theorem~\ref{thm_inv_combing} in Section~\ref{sec_proof_inv}. The integrals in Theorem~\ref{thm_inv_combing} converge because of the following results. There exist natural smooth compactifications $C_{V(\Gamma)}(M)$ with corners of the spaces $\check{C}_{V(\Gamma)}(\check{M})$ for any trivalent graph $\Gamma$. The maps $p_e$ for $e\in E(\Gamma)$ smoothly extend as maps $p_e\colon C_{V(\Gamma)}(M) \to C_2(M)$.

Let $\beta_n \in \Ac_n^c(\emptyset)$ denote the degree $n$ part of the $\beta$ anomaly, as in \cite[Definition 10.5]{lesbook}. The following theorem is proved in Section~\ref{sec_proof_var}.

\begin{thma}\label{thm_inv_var}
Let $[X]$ be a combing on $\check{M}$. For any $n\in\N$, we have
$$z_n(M,X)=z_n(M)+\frac{1}{4}p_1(X)\beta_n.$$
\end{thma}

\begin{rquea} 
In general, we cannot define a preferred class of combings $X$ such that $p_1(X)= 0$.
 There exist
rational homology spheres $M$ such that $p_1(X)\neq 0$ for all combings $[X]$ of $\check{M}$. For example $p_1(X)$ is an odd integer for any combing $[X]$ of $\check{\R P^3}$. This is a consequence of Theorem 10 of Kirby and Melvin in \cite{km} and Theorem 1.4 of Lescop in \cite{lescomb}.
\end{rquea}

\subsection{A generalisation of Theorem~\ref{thm_inv_combing} using special multiframings}

Shimizu defined and studied multisections of bundles in \cite{shifra}. Multiframings of $M$ are multisections of the bundle of framings of $TM$.  Let us introduce a notion of a  $G$-multiframing of $\check{M}$ for a finite subgroup $G$ of $SO(3)$

\begin{defia}
 Let $G$ be a finite subgroup of $SO(3)$. Let $k\geq 1$ be an integer.  Let $(U^a)_{a\in A}$ be a covering of $\check{M}$. For any $a\in A$, let $(\tau^a_{i}\colon U^a \times \R^3 \to TU^a)_{i\in\{1,\dots,k\}}$ be a family of parallelizations of $U^a$ with the following properties.
 \begin{itemize}
  \item For any $a,b$ in $A$, there exists a permutation $\sigma^{ab}$ of $\{1,\dots,k\}$ such that $\tau^a_{\sigma^{ab}(i)}= \tau^b_{i}$ on $U^a \cap U^b$ for all $i\in\{1,\dots,k\}$.
  \item For any $a\in A$ and any $i,j\in \{1,\dots,k\}$, the map $(\tau^a_{j})^{-1} \circ \tau^a_{i}$ is equal to $\mbox{Id}_{U^a} \times g_{ij}$ for a fixed element $g_{ij}$ of $G$.
  \item Let $\tau^{\mbox{\scriptsize st}}\colon \check{V}_\infty\times\R^3 \to U\check{V}_\infty$ be the standard parallelization of $U\check{V}_\infty$. For any $a\in A$ and any $j\in \{1,\dots,k\}$, the restriction of the map $(\tau^a_{j})^{-1} \circ \tau^{\mbox{\scriptsize  st}}$ to $U^a \cap V_\infty$ is equal to $\mbox{Id}_{U^a} \times g_{\mbox{\scriptsize st},j}$ for a fixed element $g_{\mbox{\scriptsize st},j}$ of $G$.
 \end{itemize}
We say that $\Tilde{\tau}=\bigl(U^a,(\tau^a_{i})_{i\in\{1,\dots,k\}}\bigr)_{a\in A}$ is a $k$-fold $G$-framing of $\check{M}$.  A $G$-multiframing of $\check{M}$ is a $k$-fold $G$-framing of $\check{M}$ for some $k$.
\end{defia}

Let us outline the definition of an invariant $z$ of a rational homology sphere equipped with a $k$-fold $G$-framing $\Tilde{\tau}=\bigl(U^a,(\tau^a_{i})_{i\in\{1,\dots,k\}}\bigr)_{a\in A}$.  Define a map $p_{\Tilde{\tau}}^*$ from the space of $G$-invariant forms on $S^2$ to the space of forms on $\partial C_2(M)$ to be the average of the $k$ pull-back maps $(p_{\tau_i^a}^*)_{i\in\{1,\dots,k\}}$ on each $U^a$. A \emph{propagating form of $(C_2(M),\Tilde{\tau})$} is a propagating form of $C_2(M)$ that restricts to $\partial C_2(M)$ as $p_{\Tilde{\tau}}^*(\omega)$ for some $G$-invariant form $\omega$ of volume one on $S^2$.

We omit the proof of the following theorem. It is similar to the proof of Theorem~\ref{thm_inv_combing}.

\begin{thma}\label{thm_inv_multiframing}
Let $G$ be a finite subgroup of $SO(3)$. Let $k\geq 1$ be an integer. Let $\Tilde{\tau}=\bigl(U^a,(\tau^a_{i})_{i\in\{1,\dots,k\}}\bigr)_{a\in A}$ be a $k$-fold $G$-framing of $\check{M}$. Let $(\omega_i)_{i\in\{1,\dots,3n\}}$ be a family of $3n$ propagating forms of $(C_2(M),\Tilde{\tau})$. The element
 $$z_n(M,(\omega_i)_{i\in\{1,\dots,3n\}})=\xi_n  \sum_{(\Gamma,j)\in \Dc_n^c(\emptyset)}  I(M,\Gamma,j,(\omega_i)_{i\in\{1,\dots,3n\}})[\Gamma]$$
 of $\Ac_n^c(\emptyset)$ does not depend on the choice of the family $(\omega_i)_{i\in\{1,\dots,3n\}}$ of propagating forms of $(C_2(M),\Tilde{\tau})$. We denote this element by $z_n(M,\Tilde{\tau})$.
\end{thma}

A combing $[X]$ on $\check{M}$ and a $p$-multisection $\nu_X^p$ of $X$ induce a $p$-fold
$(\U_p \subset SO(2) \subset SO(3))$-framing
$\Tilde{\tau}(X,\nu_X^p)$ of $\check{M}$, where $SO(2)$ is seen as the subgroup of $SO(3)$ of rotations about $(0,0,1)$. We have $p_{\Tilde{\tau}(X,\nu_X^p)}^*=p_{(X,\nu_X^p)}^*$. Thus we obtain $z(M,\Tilde{\tau}(X,\nu_X^p))=z(M,X)$.

In \cite{shifra}, Shimizu defined a $\Q$-valued map $p_1$ on the homotopy classes of multiframings on a closed oriented $3$-manifold. We ask the following natural questions. 

\begin{questions}
Does Shimizu's definition of $p_1$ extend to multiframings of $\check{M}$~? If this is the case, do we have $p_1(X)=p_1(\Tilde{\tau}(X,\nu_X^p))$ ? When $G$ is a finite subgroup of $SO(3)$ and $\Tilde{\tau}$ is a $G$-multiframing of $\check{M}$, does the variation formula $z(M,\Tilde{\tau})=z(M)+1/4 p_1(\Tilde{\tau})\beta$ hold ? 
\end{questions}

\subsection{A dual version of Theorem~\ref{thm_inv_combing}}
\label{sec_dual_main_result}

In this section, we state Theorem~\ref{thm_inv_combing_dual}, which is a dual version of Theorem~\ref{thm_inv_combing}.
 
 \begin{defia}
 Let $q\in S^2$. Define a $3$-dimensional cycle $C(M,g,X,\nu^p_{X};q)$ of $\partial C_2(M)$ as follows. Let $B$ be an embedded ball in $\check{M}$. Let $\nu_{X}$ be a lift of $\nu^p_{X}$ over $B$.
 Set $$C(M,g,X,\nu^p_{X};q) \cap U\check{M}|_{B} =\frac{1}{p} \sum_{k\in \{1,\dots,p\}} \bigl(R_{\frac{2k\pi}{p}} \circ p_{(\nu_{X},w_X,X)}\bigr)^{-1} (q).$$

 Set $$C(M,g,X,\nu^p_{X};q) \cap (\partial C_2(M) \setminus U\check{M}) =\frac{1}{p} \sum_{k\in \{1,\dots,p\}} \bigl(R_{\frac{2k\pi}{p}} \circ G_M \bigr)^{-1} (q).$$
\end{defia}

\begin{defia}
 A \emph{propagating chain of $(C_2(M),g,X,\nu^p_{X})$} is a four-dimensional relative cycle $P$ of $(C_2(M),\partial C_2(M))$ such that 
 $\partial P = C(M,g,X,\nu^p_{X};q)$ for some $q\in S^2$.\footnote{In this article, an $n$-dimensional relative cycle of a smooth manifold $X$ with boundary $\partial X$ is a formal combination $P$ of smooth oriented submanifolds of dimension $n$  of $X$ with rational coefficients such that $\partial P$ is a formal combination of smooth oriented submanifolds of dimension $n-1$  of $\partial X$ with rational coefficients.}
\end{defia}

Propagating chains of $(C_2(M),g,X,\nu^p_{X})$ exist because $H_3(C_2(M);\R)=0$. As in \cite[Lemma 11.4]{lesbook}, one can show that there exists a family $(P_i)_{i\{1,\dots,3n\}}$ of propagating chains of $(C_2(M),g,X,\nu^p_{X})$ such that for any $(\Gamma,j)\in \Dc_n^c(\emptyset)$ the intersection
$\bigcap_{e\in E(\Gamma)} P_{j(e)}$
is transverse. Say that such a family is in \emph{general position}. See also \cite[Definition 11.3]{lesbook}. For such a family and such a labelled graph $(\Gamma,j)$ equipped with a vertex-orientation $o(\Gamma)$, the algebraic intersection of the $P_{j(e)},e\in E(\Gamma)$ on $C_{V(\Gamma)}(M)$ is denoted by
$$I(M,\Gamma,j,(P_i)_{i\in\{1,\dots,3n\}})[\Gamma]= I(M,\Gamma,o(\Gamma),j,(P_i)_{i\in\{1,\dots,3n\}})[\Gamma,o(\Gamma)].$$

\begin{thma}\label{thm_inv_combing_dual}
Let $[X],g,p,\nu^p_{g,X}$ be as in Theorem~\ref{thm_inv_combing}. Let $(P_i)_{i\in\{1,\dots,3n\}}$ be a family of $3n$ propagating chains of $(C_2(M),g,X, \nu^p_{X})$ in general position. We have
 $$z_n(M,X)=\xi_n \sum_{(\Gamma,j)\in \Dc_n^c(\emptyset)} I(M,\Gamma,j,(P_i)_{i\in\{1,\dots,3n\}})[\Gamma].$$
\end{thma}

\subsection{Organisation of the paper}

The article is organised as follows. In Section~\ref{sec_compare}, we compare the invariant $z(M,[X])$ with other existing invariants. Assuming Theorem~\ref{thm_inv_combing}, we prove Theorem~\ref{thm_inv_var} for $n=1$. In Section~\ref{sec_multisection} we study multisections of $UX^\perp$ in more details. Sections \ref{sec_proof_inv} and \ref{sec_proof_var} are respectively devoted to the proofs of Theorems \ref{thm_inv_combing} and \ref{thm_inv_var}.

\section{Relation with other invariants}

\label{sec_compare}

In this section, we assume that Theorem~\ref{thm_inv_combing} is proven. In Section~\ref{sec_relate}, we recall Lescop's definition of the invariant $z$ of a rational homology sphere  from \cite{lesbook}. In Section~\ref{sec_relate_theta}, we define the Theta invariant of a rational homology sphere equipped with a combing following \cite{lescomb}. We prove Theorem~\ref{thm_inv_var} in the special case where $n=1$. In \cite{shimizu}, Shimizu defined an invariant $\Tilde{z}$ of rational homology spheres, using the additional datum of a section of the tangent bundle to the manifold that is transverse to the zero section. We compare the two constructions in Section~\ref{sec_relate_Shimizu}.

\subsection{The invariant \texorpdfstring{$z$}{z} of a rational homology sphere}
\label{sec_relate}

The following theorem is a particular case of \cite[Theorem 7.20]{lesbook}.

\begin{thma}\label{thm_inv_tau}
Let $n\geq 1$ be an integer. Let $\tau$ be a parallelization of $\check{M}$. Let $(\omega_i)_{i\in\{1,\dots,3n\}}$ be a family of $3n$ propagating forms of $(C_2(M),\tau)$. The element
 $$z_n(M,(\omega_i)_{i\in\{1,\dots,3n\}})=\xi_n\sum_{(\Gamma,j)\in \Dc_n^c(\emptyset)}  I(M,\Gamma,j,(\omega_i)_{i\in\{1,\dots,3n\}})[\Gamma]$$
 of $\Ac_n^c(\emptyset)$ does not depend on the choice of the family $(\omega_i)_{i\in\{1,\dots,3n\}}$ of propagating forms of $(C_2(M),\tau)$. It only depends on the homotopy class of $\tau$ and on the orientation-preserving diffeomorphism type of $M$. We denote this element by $z_n(M,\tau)$.
 
 The element $z(M,\tau)-p_1(\tau)/4 \beta$ does not depend on the chosen parallelization $\tau$. Define $z(M)$ as $z(M)=z(M,\tau)-p_1(\tau)/4 \beta$.
\end{thma}

Assuming Theorem~\ref{thm_inv_combing}, the next proposition follows easily.

\begin{propa} \label{prop_eq_para}
 Let $[X]$ be a combing on $\check{M}$ that extends to a parallelization $\tau$ of $\check{M}$. Then we have
 $$z_n(M,X)=z_n(M,\tau).$$
\end{propa}
\begin{proof}
 Assume that $X=\tau(\cdot,(0,0,1))$. Let $g$ be the Riemannian metric on $\check{M}$ which restricts to the standard metric on $\R^3$ on each fiber, with respect to the identification given by $\tau$.  Let $\nu_{X}$ denote the ($1$-multi)section $\tau(\cdot,(1,0,0))$ of $UX^\perp$. Then a propagating form of $(C_2(M),g,X,\nu_{X})$ is also a propagating form of $(C_2(M),\tau)$. The result follows.
\end{proof}
Theorem~\ref{thm_inv_tau} and Proposition~\ref{prop_eq_para} imply Theorem~\ref{thm_inv_var} when $X$ extends to a parallelization $\check{M}$, which always holds when $M$ is an integer homology sphere.

\subsection{The Theta invariant of a rational homology sphere equipped with a combing}
\label{sec_relate_theta}

In Theorem~\ref{thm_theta_combing}, we give an alternative definition for the invariant $\Theta$ of rational homology spheres equipped with a combing studied in \cite{lescomb}. We prove Theorem~\ref{thm_inv_var} when $n=1$ as a direct consequence of Theorem 5.3 of Lescop in \cite{lescomb}. See Proposition~\ref{prop_eq_theta_inv}.

We define a preferred generator of $H_2(C_2(M);\Q)$ as follows. Orient $U\check{M}$ as part of the boundary of $C_2(M)$. Orient the $S^2$ fibers of $U\check{M}$ so that this orientation followed by the orientation of $\check{M}$ is the orientation of $U\check{M}$. Our preferred generator of $H_2(C_2(M);\Q)$ is the class $[S]$ of a fiber $S$ of $U\check{M}$. For any propagating chain $\mathcal{P}$ of $C_2(M)$, we have $\langle \mathcal{P},[S]\rangle_{C_2(M)}=1$.

\begin{thma}\label{thm_theta_combing}
   Let $[X]$ be a combing on $\check{M}$. Let $p\in\N\setminus\{0\}$ be a multiple of $\order(X)$. Let $\nu_X^p$ be a $p$-multisection of $UX^\perp$. Let $q_a,q_b$ be two points in $S^2$  such that $q_b$ is not in the orbit of $q_a$ for the $\U_p$ action on $S^2$. Let $P_a$ be a propagating chain of $(C_2(M),g,X,\nu_X^p)$ such that $\partial \mathcal{P}_a = C(M,g,X,\nu_X^p;q_a)$. Let $P_b$ be a propagating chain of $(C_2(M),g,X,\nu_X^p)$ such that $\partial \mathcal{P}_b = C(M,g,X,\nu_X^p;q_b)$. We can choose $\mathcal{P}_b$ so that $\mathcal{P}_a$ and $\mathcal{P}_b$ intertersect transversely. Then $\mathcal{P}_a \cap \mathcal{P}_{b}$  is a $2$-dimensional cycle of $C_2(M)$. Its homology class is independent on the chosen propagating chains. We denote it by $\Theta(M,X)[S]$. 
\end{thma}
\begin{proof}
Denote the homology class of $\mathcal{P}_a \cap \mathcal{P}_{b}$ by $\Theta(M,X,P_a,P_b)[S]$. Let $q_c$ be a point in $S^2$ that is not contained in the orbits of $q_a$ and $q_b$ for the $\U_p$ action on $S^2$. Let $\mathcal{P}_c$ be a propagating chain of $(C_2(M),g,X,\nu_X^p)$ such that $\partial \mathcal{P}_c= C(M,g,X,\nu_X^p;q_c)$. Without loss of generality, we assume that $\mathcal{P}_a\cap \mathcal{P}_b \cap \mathcal{P}_c$ and all pairwise intersections between $\mathcal{P}_a$, $\mathcal{P}_b$ and $\mathcal{P}_c$ are transverse. Then we have
$$\Theta(M,X,\mathcal{P}_a,\mathcal{P}_b)=\langle \mathcal{P}_a,\mathcal{P}_b,\mathcal{P}_c \rangle_{C_2(M)}= \langle \mathcal{P}_a,\mathcal{P}_c,\mathcal{P}_b \rangle_{C_2(M)}=\Theta(M,X,\mathcal{P}_a,\mathcal{P}_c).$$
This shows that $\Theta(M,X,\mathcal{P}_a,\mathcal{P}_b)$ does not depend on the choices of $q_b$ and $P_b$. Similarly, it does not depend on the choices of $q_a$ and $P_a$.
\end{proof}

 There is a natural smooth involution of $\check{C}_2(\check{M})$ that exchanges the images of the two points. It extends as a smooth involution $\iota$ of $C_2(M)$. 
 
 \begin{faita}\label{fact_iota}
  Let $[X]$ be a combing on $\check{M}$. Let $p\in\N\setminus\{0\}$ be a multiple of $\order(X)$. Let $\nu_X^p$ be a $p$-multisection of $UX^\perp$. Let $\mathcal{P}$ be a propagating chain of $(C_2(M),g,X,\nu_X^p)$. Then $\iota(\mathcal{P})$
  is also a propagating chain of $(C_2(M),g,X,\nu_X^p)$.
\end{faita}

The next proposition expresses the relation between the invariants $z_1(M,X)$ and $\Theta(M,X)$.

\begin{propa}
 We have
 $$z_1(M,X) =\frac{1}{12} \Theta(M,X) \left[\tata \right].$$
\end{propa}
\begin{proof}
Let $P_1,P_2,P_3$ be three propagating chains of $(C_2(M),g,X,\nu_X^p)$ in general position.  
 According to Theorem~\ref{thm_inv_combing_dual}, we have
  $$z_1(M,X)=\xi_1 \sum_{(\Gamma,j)\in \Dc_1^c(\emptyset)} I(M,\Gamma,j,(P_i)_{i\in\{1,2,3\}})[\Gamma].$$
  We have $\xi_1=1/48$. There are $4$ elements of $\Dc_1^c(\emptyset)$. As a consequence of Theorem~\ref{thm_theta_combing} and Fact~\ref{fact_iota}, for every $(\Gamma,j)\in \Dc_1^c(\emptyset)$ we have $I(M,\Gamma,j,(P_i)_{i\in\{1,2,3\}})[\Gamma]= \Theta(M,X)[\tata]$. The result follows.
\end{proof}

In the particular case of Theorem~\ref{thm_theta_combing} where $q_a=(0,0,1)$ and $q_b=(0,0,-1)$, we recover Lescop's definition of the Theta invariant of a rational homology sphere equipped with a combing. See \cite[Theorem 5.1]{lescomb}.
The next proposition is equivalent to \cite[Theorem 5.3]{lescomb}. According to \cite[Proposition 10.10]{lesbook}, we have $\beta_1=1/12 [\tata]$.

\begin{propa}\label{prop_eq_theta_inv}
 Let $[X]$ be a combing on $\check{M}$. We have
 $$z_1(M,X)=z_1(M)+\frac{1}{4}p_1(X)\beta_1.$$
\end{propa}
 
\subsection{A definition of \texorpdfstring{$z$}{z} via vector fields}

\label{sec_relate_Shimizu}

In \cite{shimizu}, Shimizu defined \emph{admissible vector fields} on $\check{M}$ as sections of $T\check{M}$ that are transverse to the zero section and standard near $\infty$. He used them to constrain propagators on the boundary of the configuration space of two points in $M$. Shimizu studied the dependence on these constraints of a function $\Tilde{z}_n(M,(\omega_i)_{i\in\{1,\dots,3n\}})$ of a family $(\omega_i)_{i\in\{1,\dots,3n\}}$ of such propagating forms. He obtained a correcting term $\Tilde{z}_n^{\mbox{\scriptsize anomaly}}(M,(\omega_i)_{i\in\{1,\dots,3n\}})$ associated with an appropriate $3$-bundle over a $4$-manifold
whose boundary is $M$. 
Shimizu proved the equality
$$Z(M)=\Tilde{z}_n(M,(\omega_i)_{i\in\{1,\dots,3n\}})-\Tilde{z}_n^{\mbox{\scriptsize anomaly}}(M,(\omega_i)_{i\in\{1,\dots,3n\}}).$$
Gradients of \say{admissible Morse functions} on $\check{M}$ are examples of admissible vector fields. See \cite[Example 4.2]{shimizu}. This allowed Shimizu to prove the equivalence between $Z$ and Watanabe's Morse homotopy invariant defined in \cite{watanabeMorse}.

In general, both
$\Tilde{z}_n(M,(\omega_i)_{i\in\{1,\dots,3n\}})$ and $\Tilde{z}_n^{\mbox{\scriptsize anomaly}}(M,(\omega_i)_{i\in\{1,\dots,3n\}})$ depend on the forms $\omega_i$. In \cite[Remark 4.7]{shimizu}, Shimizu stated that these elements do not depend on the forms $\omega_i$ when they are adapted to  a generic family $(\gamma_i)_{i\in\{1,\dots,3n\}}$ of admissible vector fields.  Combings are examples of admissible vector fields.  When $X$ is a combing, the constant family equal to $X$ is not generic. We therefore ask the following question.

\begin{questions}
\label{ques_shimizu}
 Can Shimizu's construction be adapted to give a consistent definition of an invariant of a rational homology sphere equipped with a combing ? Could such an invariant be related to the invariant $z(M,X)$ of Theorem~\ref{thm_inv_combing} ?
\end{questions}

Recall the involution $\iota$ of $C_2(M)$ from the previous section. Propagating forms that are transformed into their opposite by this involution are called \emph{antisymmetric}. In Shimizu's construction the restriction on $\partial C_2(M)$ of propagators adapted to an admissible vector field $\gamma$ must be antisymmetric, even when $\gamma$ is nowhere vanishing. In contrast, our definition of $Z$ allows the use of non antisymmetric propagators. As explained in \cite{Man}, such a use may make concrete computations of $Z$ easier. For example, Lescop used non antisymmetric propagating chains to compute $\Theta(M,X)$ in \cite{lesHC}.

\section{About multisections}
\label{sec_multisection}
Throughout this section, $[X]$ is a fixed combing on $\check{M}$. We study multisections of $UX^\perp$. In Section~\ref{sec_Euler}, we give a definition of the Euler class of a principal $S^1$-bundle. In Section~\ref{sec_concrete_multisection}, we construct a special multisection of $UX^\perp$ that we will use in our proof of Theorem~\ref{thm_inv_var}.

Let $p\in\N\setminus \{0\}$.
Let $E$ be an $S^1$-principal bundle on $\check{M}$. Let $E_p$ denote the quotient $E/\mathbb{U}_p$ with respect to the action of $\U_p\subset S^1$ on $E$. In this paragraph, we give some details about the induced structure of an $S^1$-principal bundle on $E_p$. Let $f_p\colon E \to E_{p}$ denote the quotient map. The $S^1$-action on $E$ induces a free and transitive $S^1/\U_p$-action on $E_{p}$. The map $z\mapsto z^p$ from $S^1$ to $S^1$ induces an isomorphism of Lie groups from $S^1/\U_p$ to $S^1$. Using the above identification, the $S^1/\U_p$-action on $E_{p}$ becomes an $S^1$-action that can be described as follows. Let $(z,x)\in S^1 \times E_{p}$. Let $(\Tilde{z},\Tilde{x})\in S^1 \times E$ be such that $\Tilde{z}^p=z$ and $f_p(\Tilde{x})=x$. Then $z\cdot x =f_p(\Tilde{z} \cdot \Tilde{x})$. This action equips $E_{p}$ with the structure of a principal $S^1$-bundle. The map $f_p$ is a morphism of principal $S^1$-bundles, and a covering of degree $p$.

\begin{lema}\label{lem_exist_section}
Let $p\in\N\setminus\{0\}$ be a multiple of $\order(X)$. There exists a $p$-multisection of $UX^\perp$.
\end{lema}
 We give a reformulation of Lemma~\ref{lem_exist_section} in terms of the vanishing of an Euler class in Lemma~\ref{lem_Euler_knot}. We prove Lemma~\ref{lem_exist_section} in Section~\ref{sec_concrete_multisection} after Lemma~\ref{lem_concrete_multi}. We now prove that any two $p$-multisections of $UX^\perp$ are homotopic.

\begin{lema}\label{lem_multi_hom}
 Let $p\in\N\setminus\{0\}$ be a multiple of $\order(X)$. Let $\nu^p_X$ and $w^p_X$ be two $p$-multisections of $UX^\perp$. There exists a smooth homotopy $H\colon [0,1]\times \check{M} \to UX^\perp_p$ between $\nu^p_X$ and $w^p_X$ through $p$-multisections of $UX^\perp$.
\end{lema}
\begin{proof}
 The section $\nu_X^p$ induces a bundle isomorphism between $UX^\perp_p$ and $\check{M} \times S^1$ such that $\nu_X^p$ is identified with the constant section equal to $1$. Using this identification, the section $w_X^p$
 becomes
 a map $w_X^p\colon \check{M} \to S^1$ that is constant equal to $1$ on $\check{V}_\infty$. Extend this map
 to 
 a map from $M$ to $S^1$ constant equal to $1$ on $V_\infty$. To prove the lemma, it suffices to show that
 any such map is homotopic to the constant map equal to $1$ relatively to $V_\infty$. As $V_\infty$ is an embedded $3$-ball in $M$, the set of homotopy classes of maps from $(M,V_\infty)$ to $(S^1,1)$ is in bijection with $H^1(M,\Z)$. As $M$ is a rational homology sphere, we have $H^1(M,\Z)=0$. The result follows. 
\end{proof}

\subsection{The Euler class of a principal circle bundle}
\label{sec_Euler}
Let $E$ be an $S^1$-principal bundle on $\check{M}$ equipped with a trivialisation $E|_{\check{V}_\infty}\to S^1 \times \check{V}_\infty$ over $\check{V}_\infty$. We associate to $E$ a smooth $D^2$-bundle $D(E)$ over $\check{M}$ whose transition maps are the transition maps of $E$ seen as elements of $S^1=SO(2) \subset \mbox{Diff}^+(D^2)$.
We use the outward normal first convention to orient boundaries.
The bundle $D(E)$
is equipped with a trivialisation over $V_\infty$. The zero section is a well-defined smooth section of $D(E)$. 
We use the next lemma as a definition for the Euler class of $E$.

\begin{lema}\label{lem_Euler}
Let $E$ be an $S^1$-bundle on $\check{M}$ equip with a trivialisation $E|_{V_\infty}\to S^1 \times \check{V}_\infty$ over $\check{V}_\infty$. There exists a smooth section $Y\colon \check{M} \to D(E)$ of $D(E)$ which is transverse to the zero section and constant equal to $1$ on $V_\infty \setminus\{\infty\}$. The set $L_Y=\{m\in\check{M}\suchthat Y(m)=0\}$ is a smooth link in $M\setminus V_\infty$. It is cooriented by the $D^2$ fibers of $D(E)$. The homology class $[L_Y]$ of $L_Y$ does not depend on the chosen section $Y$. It is by definition the Poincaré dual to the Euler class $e(E)$ of $E$.
\end{lema}
\begin{proof}

We only prove that the homology class of $L_Y$ does not depend on $Y$. Assume that $Y_0$ and $Y_1$ are two such sections of $D(E)$. The maps $H_t=(1-t)Y_0+tY_1$ for $t\in [0,1]$ define a smooth homotopy from $Y_0$ to $Y_1$ that is constant equal to $1$ on $[0,1]\times (V_\infty \setminus\{\infty\})$. The homotopy $H$ can be perturbed on $]0,1[\times \check{M}$ so that the set $S_H=\{(t,m)\in [0,1]\times \check{M}\suchthat H_t(m)=0\}$ is a smooth proper compact surface in $[0,1]\times \check{M}$. The surface $S_H$ is cooriented by the $D^2$ fibers of $D(E)$. The image of $S_H$ by the projection from $[0,1]\times \check{M}$ onto the factor $\check{M}$ is a $2$-chain of $\check{M}$ whose boundary is $L_{Y_1}-L_{Y_0}$.
\end{proof}

In the next definition, we explain how to twist sections of a principal $S^1$-bundle along the normal direction to an embedded surface.

\begin{defia}\label{def_twist}
 Let $N$ be a smooth compact connected $3$-manifold. Let $E$ be a principal $S^1$-bundle over $N$. Let $Y$ be a section of $E$ over $N$. Let $\Sigma$ be a properly embedded surface in $N$. Choose a smooth orientation-preserving embedding of $[-1,1]\times \Sigma$ into $N$ that maps $[-1,1]\times \partial \Sigma$ into $\partial N$ and such that $\{0\}\times \Sigma$ is identified with $\Sigma$. Let $u$ denote
 the $[-1,1]$-coordinate. 
 Choose a smooth increasing function $\kiS{-1}{1}\colon [-1,1]\to [0,1]$ that is $0$ near $-1$ and $1$ near $1$. Let $k\in\Z$. The 
 \emph{$k$-twist}
 of $Y$ along $\Sigma$ is the section $Y(\Sigma,k)$ of $E$ over $N$ defined to be $Y(\Sigma,k)=Y$ on $N \setminus ([-1,1]\times \Sigma)$ and
 $$Y(\Sigma,k)=\exp(2ik\pi \kiS{-1}{1}(u))\cdot Y$$
 on $[-1,1]\times \Sigma$.
\end{defia}

In the next definition, we introduce the relative degree of two sections of a principal $S^1$-bundle over $S^1$.

\begin{defia}\label{def_degree}
 Let $\gamma$ be a smooth closed connected $1$-manifold. Let $E$ be a principal $S^1$-bundle on $\gamma$. Let $s$ be a section of $E$. It induces a trivialisation $E \to S^1\times \gamma$ of $E$. Any section $s'$ of $E$ can be seen as a map from $\gamma$ to $S^1$ using this trivialisation. Let $d_{\gamma}(s,s')$ denote the degree of this map.
\end{defia}

We dot not prove the following four easy lemmas.

\begin{lema}\label{lem_deg_sum}
Let $\gamma,E,s,s'$ be as in Definition \ref{def_degree}. If $s''$ is a third section of $E$, we have $d_{\gamma}(s,s'')=d_{\gamma}(s,s')+d_{\gamma}(s',s'')$.
\end{lema}

\begin{lema}\label{lem_deg_mer}
 Let
$E$, $Y$, and $L_Y$ 
be as in Lemma~\ref{lem_Euler}. Let $N(L_Y)$ be a tubular neighbourhood of $L_Y$. Let $m \subset \partial N(L_Y)$ be a meridian of a connected component of $N(L_Y)$. Orient $m$ such that $\lk(m,L_Y)=1$. Let $s$ be a section of $E$ over a meridian disk $D_m \subset N(L_Y)$ such that $\partial D_m =m$. Then $d_m(s,Y)=1$.
\end{lema}

\begin{lema}\label{lem_deg_ext}
Let $E$ be a principal $S^1$-bundle of a solid torus $T=D^2 \times S^1$. Let $s$ be a section of $E$ over a meridian disk $D^2  \times \{1\}$ of $T$. Let $s'$ be a section of $E$ over $\partial T$. Then $s'$ extends as a section of $E$ over $T$ if and only if $d_{(\partial D^2) \times \{1\}}(s,s')=0$.
\end{lema}

The following lemma describes the  effect of twisting a section of a principal $S^1$-bundle on the degree.

\begin{lema}\label{lem_deg_twist}
 Let $N,E,Y,\Sigma,k$ be as in Definition \ref{def_twist}. Let $\gamma$ be an oriented closed simple curve on $\partial N$. Consider a section $s$ of $E|_{\gamma}$. We have $ d_{\gamma}(s,Y(\Sigma,k))=d_{\gamma}(s,Y) + k\langle\partial \Sigma,\gamma \rangle_{\partial N} $, where $\langle \cdot, \cdot \rangle_{\partial N}$ denotes the algebraic intersection over $\partial N$. (Here $\partial N$ is oriented as the boundary of $N$.)
\end{lema}

\begin{lema}\label{lem_Euler_knot}
Let $E$ be as in Lemma~\ref{lem_Euler}. Let $K$ be a smooth knot in $M\setminus V_\infty$ that represents the Poincaré dual to $e(E)$. There exists a section $\Tilde{Y}$ of $D(E)$ equal to $1$ on $V_\infty \setminus\{\infty\}$ such that $L_{\Tilde{Y}}=K$.

If $e(E)=0$, then there exists a smooth section of $E$ which is constant equal to $1$ on $V_\infty \setminus\{\infty\}$.
\end{lema}
\begin{proof}
 Let $Y$ be a section of $X^\perp$ as in Lemma~\ref{lem_Euler}. Let $N(L_Y)$ and $N(K)$ be closed tubular neighbourhoods of $L_Y$ and $K$ in $M\setminus V_\infty$. Let $N$ denote the closure of $M\setminus (N(L_Y) \cup N(K) \cup V_\infty)$ in $M$. As $[K]=[L_Y]$ in $H_1(M,\Z)$, there exists parallels $K_{//}\subset \partial N(K)$ and $(L_Y)_{//}\subset \partial N(L_Y)$ of $K$ and $L_Y$ and a smooth oriented surface $S$ properly embedded in $N$ such that $\partial S = K_{//} -(L_Y)_{//}$. When $e(E)=0$, choose $S$ such that $\partial S = -(L_Y)_{//}$.
 
 Let $[-1,1] \times S \to N$ be a smooth orientation-preserving embedding which restricts as the identity from $\{0\} \times S$ to $S$ and maps $[-1,1]\times \partial S$ to $\partial N$. Consider the section $Y(S,1)$ of $E$ over $N$.  It can be extended over $N(L_Y)$ as a nowhere vanishing section of $D(E)$. The result follows when $e(E)=0$. When this is not the case, the section $Y(S,1)$ can be further extended as a section of $D(E)$ over $N(K)$ so that $L_{Y(S,1)}=K$.
 \end{proof}

\begin{rquea}\label{rque_Euler_quotient}
 Let $E$ be as in \ref{lem_Euler}. Let $p\in\N\setminus\{0\}$. Equip $E_p$ with the induced trivialisation over $\check{V}_\infty$. One could easily show that $e(E_p)=pe(E)$ and use this fact
 to prove Lemma~\ref{lem_exist_section}.
\end{rquea}

\subsection{A multisection construction}
\label{sec_concrete_multisection}

Let $p\in\N\setminus\{0\}$ be a multiple of $\order(X)$. In this section we construct a $p$-multisection of $UX^\perp$ and prove Lemma~\ref{lem_exist_section}. We will use this construction in the proof of Theorem~\ref{thm_inv_var}. The following lemma is a direct corollary of Lemma~\ref{lem_Euler_knot}. See Figure~\ref{fig_setting}.

\begin{lema}\label{lem_exist_knot}
 Let $a,b,\epsilon$ be real numbers with $a<b$ and $\epsilon < \max(1,b-a)/4$. There exists an oriented knot $K\subset M\setminus V_\infty$ and a neighbourhood $N(K) \subset M\setminus V_\infty$ of $K$ equipped with an orientation-preserving diffeomorphism $\Psi \colon [a,b]\times S^1\times [-1,1] \to N(K)$ with the following properties.
 \begin{itemize}
 \item The element $[K]$ of $H_1(M;\Z)$ is Poincaré dual to $-e(UX^\perp)$.\footnote{We take $-e(UX^\perp)$ instead of $e(UX^\perp)$ for technical reasons.}
  \item The section $X$ extends as an orthonormal framing $(Y_e,W_e,X)$ of $T\check{M}$ over $E_K=\check{M}\setminus \Psi(]a+\epsilon,b-\epsilon[\times S^1\times ]-1+\epsilon,1-\epsilon[)$ that is the standard framing of $\R^3$ on $\check{V}_\infty$. Let $\tau_e\colon E_K \times \R^3 \to TE_K$ denote the induced parallelization of $E_K$. (The letter \say{e} stands for \say{exterior}.)
  
 \end{itemize}
\end{lema}

Let $D_m$ be the meridian disk of $N(K)$ identified with $[a,b] \times \{1\} \times [-1,1]$, and oriented as $- [a,b] \times [-1,1]$. Let $m_K\subset \partial N(K)$ denote the oriented curve $\partial D_m$. Let $l_K\subset \partial N(K)$ denote the curve identified with $\{b\} \times S^1 \times \{0\}$, and oriented as $S^1$. When $s$ is a section of $UX^\perp$ over $D_m$, we have $d_{m_K}(s,Y_e)=-1$.

The \emph{linking number} $\lk(L_1,L_2)\in\Q$ of two disjoints links $L_1$ and $L_2$ in $M$ is the algebraic intersection of $L_1$ and a $2$-chain $S$ with rational coefficients whose boundary is $L_2$. A \emph{parallel} of a link $L$ is a link disjoint from $L$ and isotopic to $L$ in a tubular neighbourhood of $L$ in $M$. The \emph{self-linking number} $\slk(a)\in \Q/\Z$ of a class $a\in H_1(M,\Z)$ is defined
to be 
the class modulo $\Z$ of the linking number of a link representative $L_1$ of $a$ and a parallel $(L_1)_{//}$ of $L_1$.

\begin{lema}\label{lem_exist_surface}
There exists a compact surface $\Sigma$ and a proper embedding of $[-1,1]\times \Sigma$ into $E_K\setminus V_\infty$ such that the homology class of $ (\{0\}\times \Sigma) \cap \partial N(K)$ in $H_1(\partial N(K);\Z)$ is $p[l_K]+q[m_K]$ for some $q\in\Z$.  The equality $\slk(K)=-q/p$ holds in $\Q/\Z$. Let $N(\Sigma)$ denote the image of $[-1,1]\times \Sigma$.
\end{lema}
\begin{proof}
 We have $p [K]=0$ in $H_1(M;\Z)$. Thus there exists a $2$-chain $S$ of $M$ such that $\partial S = p K$. We can and do assume that $S \cap E_K$ is a surface $\Sigma$ properly embedded in $E_K$, which does not meet $V_\infty$.
 \end{proof}

Throughout the article we assume that $|q|<p$ without loss.

Let $Y^p_e=f_p \circ Y_e$. Recall Definition \ref{def_twist}. Let $u$ denote the projection from $N(\Sigma)$ onto $[-1,1]$. Let $\Tilde{Y}^p_e$ denote the section of $UX^\perp_p$ on $E_K$ defined to be $\Tilde{Y}^p_e=Y^p_e(\Sigma,1)$.

\begin{lema}\label{lem_concrete_multi}
 The restriction of the multisection $\Tilde{Y}^p_e$ to $N(K) \cap E_K$ extends as a $p$-multisection $\Tilde{Y}^p_N$ of $UX^\perp$ over $N(K)$. 
\end{lema}
\begin{proof}
Let $s$ be a section of $UX^\perp_p$ over $D_m$. We have $d_{m_K}(s,Y_e)=-1$. Thus we have $d_{m_K}(s,Y^p_e)=-p$. Applying Lemma~\ref{lem_deg_twist}, we get $d_{m_K}(s,\Tilde{Y}^p_e)=-p+p=0$. The result then follows from Lemma~\ref{lem_deg_ext}.
\end{proof}
This concludes the proof of Lemma~\ref{lem_exist_section}.

We use Lemma~\ref{lem_concrete_multi} as a definition of a $p$-multisection $\Tilde{Y}^p_N$ of $UX^\perp$ over $N(K)$ when $\slk(K)\neq 0$.

Let us now assume $\slk(K) = 0$ until the end of the section. In this case, we constrain the intersection $N(\Sigma)\cap N(K)$ to satisfy a factorisation property for later use, as in the following lemma. Again, we refer to Figure~\ref{fig_setting} for a schematic of our setting.

\begin{lema}\label{lem_exist_surface_bis}
Assume that $\slk(K)=0$. There exists a compact surface $\Sigma$ and a proper embedding of $[-1,1]\times \Sigma$ into $E_K\setminus V_\infty$ such that
$$ N(\Sigma) \cap N(K) =[a,a+\epsilon] \times S^1 \times\Bigl(\bigcup_{k\in\{1,\dots,p\}}  I_k\Bigr)$$
for a certain collection of pairwise disjoint segments $I_k$ of $[-1+\epsilon,1-\epsilon]$. 
\end{lema}

In the next definition, we lift the restriction of $\Tilde{Y}^p_e$ over $N(K) \cap E_K$ to a section of $UX^\perp$ over
$N(K) \cap E_K$.

\begin{defia}
  Assume that $\slk(K)=0$. We define a section  $\Tilde{Y}_N$ of $UX^\perp$ over $N(K) \cap E_K$ as follows. Let $k\in\{1,\dots,p\}$. On $[a,a+\epsilon]\times S^1 \times I_k \subset N(K)\cap N(\Sigma)$, we set
  $$\Tilde{Y}_N= R_{\frac{2\pi}{p}\kiS{-1}{1}(u)} \circ R_{\frac{2(k-1)\pi}{p}} \circ Y_e.$$
  Let $C_1$ denote the connected component of $(N(K) \cap E_K) \setminus N(\Sigma)$ that contains $(b,1,0)$. Number the connected components of $(N(K) \cap E_K) \setminus N(\Sigma)$ with integers in $\{1,\dots,p\}$ starting from $C_1$ ang turning around $N(K)\cap E_K$ according to the orientation of $m_K$. On $C_k$, we set
  $$\Tilde{Y}_N= R_{\frac{2(k-1)\pi}{p}} \circ Y_e.$$
\end{defia}

We use the next lemma as a definition of a parallelization $\tau_N^{\mbox{\scriptsize multi}}$ of $N(K)$.

\begin{lema}\label{lem_para_int}
   Assume that $\slk(K)=0$. The section $\Tilde{Y}_N$ of $UX^\perp$ over $N(K) \cap E_K$ extends as a section of $UX^\perp$ over $N(K)$. We again denote this last section by $\Tilde{Y}_N$. 
  
  Complete $(X,\Tilde{Y}_N)$ into an orthonormal framing $(X,\Tilde{Y}_N,W_N)$ of $TN(K)$. Let $\tau_N^{\mbox{\scriptsize multi}}$ be the parallelization of $N(K)$ associated with $(\Tilde{Y}_N,W_N,X)$.
 \end{lema}

\begin{proof}
 Let $s$ be a section of $UX^\perp$ over $D_m$. We have $d_{m_K}(s,Y^e)=-1$. We compute $d_{m_K}(Y^e,\Tilde{Y}_N)=1$. Lemma~\ref{lem_deg_sum} implies $d_{m_K}(s,\Tilde{Y}_N)=0$. The result follows from Lemma~\ref{lem_deg_ext}.
 \end{proof}

 We define a $p$-multisection $\Tilde{Y}^p_N$ of $UX^\perp$ over $N(K)$ as $\Tilde{Y}^p_N = f_p \circ \Tilde{Y}_N$ when $\slk(K)=0$.

 Let $\Tilde{Y}^p$ denote the globally defined $p$-multisection of $X^\perp$ that restricts as $\Tilde{Y}_e^p$ on $E_K$ and as $\Tilde{Y}_N^p$ on $N(K)$. 

\begin{figure}[h]
 \centering
 \begin{tikzpicture}
  \useasboundingbox (-4,-4) rectangle (4,4);
  
  \fill[lightgray!50] (-3,-3) rectangle (3,3);
   \fill[lightgray!70] (-4,-2) rectangle (-2.5,-1);
     \fill[lightgray!70] (-4,1) rectangle (-2.5,2);
     
    \draw (-4.25,1.5) node{\scriptsize $\Sigma$};
    \draw (-3.5,1.5) node[below]{\scriptsize $N(\Sigma)$};
  
  \draw (-3,-3) rectangle (3,3);
  \draw[dashed] (-3,-2.5)--(3,-2.5);
   \draw[dashed] (-3,2.5)--(3,2.5);
   \draw[dashed] (-2.5,-3)--(-2.5,3);
    \draw[dashed] (2.5,-3)--(2.5,3);
    
\draw[dashed] (-4,-2)--(-2.5,-2);
\draw (-4,-1.5)--(-2.5,-1.5);
\draw[dashed] (-4,-1)--(-2.5,-1);

\draw[dashed] (-4,2)--(-2.5,2);
\draw (-4,1.5)--(-2.5,1.5);
\draw[dashed] (-4,1)--(-2.5,1);

\draw (0,0) node{$\times$} node[below]{\scriptsize $K$};

\draw (-3,-3) node[below]{\scriptsize $a$};
\draw (-2.5,-3) node[below]{\scriptsize $a+\epsilon$};
 \draw (2.5,-3) node[below]{\scriptsize $b-\epsilon$}; 
 \draw (3,-3) node[below]{\scriptsize $b$};
 
 \draw (3,-3) node[right]{\scriptsize $-1$};
 \draw (3,-2.5) node[right]{\scriptsize $-1+\epsilon$};
  \draw (3,2.5) node[right]{\scriptsize $1-\epsilon$};
   \draw (3,3) node[right]{\scriptsize $1$};

   \draw[very thick] (3,-2)--(3,-1) node[at start]{$-$} node[at end]{$-$};
    \draw (3,-1.5) node[right]{\scriptsize$I_1$};
 \draw[very thick] (3,2)--(3,1) node[at start]{$-$} node[at end]{$-$};
    \draw (3,1.5) node[right]{\scriptsize$I_p$};
\draw (-3.5,0) node[rotate=90]{$\cdots$};

\draw[->] (4,0)--(4.5,0) node[at end, below]{\scriptsize$Y_e$};
\draw[->] (4,0)--(4.25,0.25) node[at end, right]{\scriptsize$W_e$};
\draw[->] (4,0)--(4,0.5) node[at end, left]{\scriptsize$X$};
\draw (5.25,0) node{\scriptsize $=\tau^e$};

\begin{scope}[xshift=-4cm,yshift=-2cm]
\draw[->] (4,0)--(4.5,0) node[at end, below]{\scriptsize$\Tilde{Y}_N$};
\draw[->] (4,0)--(4.25,0.25) node[at end, right]{\scriptsize$W_N$};
\draw[->] (4,0)--(4,0.5) node[at end, left]{\scriptsize$X$};
\draw (5.25,0) node{\scriptsize $=\tau_N^{\mbox{multi}}$};
\end{scope}
 \draw (1.5,-1.5) node{\scriptsize $N(K)$};
 
\draw (-4.75,-1.5) node{$\cdots$};
\draw (-4.75,1.5) node{$\cdots$};

\draw[->] (-5.25,1)-- (-5.25,2) node[midway, above, sloped]{\scriptsize$u\in [-1,1]$};
\draw[->] (-5.25,-2)-- (-5.25,-1) node[midway, above, sloped]{\scriptsize$u\in [-1,1]$};
 \draw[<-] (0.5,0) arc (0:270:0.5);
 \draw (5,1) node{\scriptsize $E_K$};
 
 \draw (0,3) node{ $>$} node[above]{\scriptsize $m_K$};
\draw (3,0) node{ $\times$} node[right]{\scriptsize $l_K$};
 \end{tikzpicture}
\caption{Our setting when $\slk(K)=0$.}
\label{fig_setting}
\end{figure}
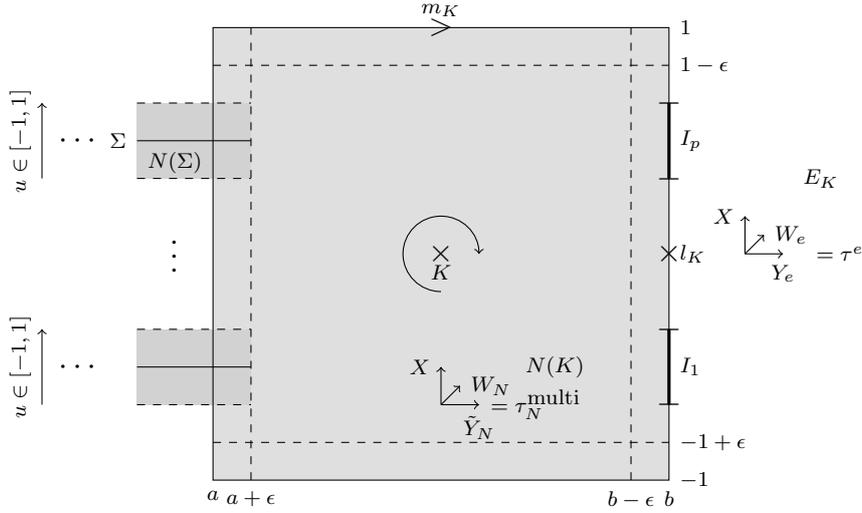

\section{Invariance of \texorpdfstring{$z(M,X)$}{z(M,X)}}
 \label{sec_proof_inv}
In this section, we prove the invariance of $z(M,X)$ stated in Theorem~\ref{thm_inv_combing}.

 \subsection{Smooth paths of parameters}
 
 Let 
 $g_0$ and $g_1$ 
 be two Riemannian metrics on $\check{M}$. Let $g\colon [0,1] \times M \to S^2 T^*M$ be the smooth path of Riemannian metrics on $\check{M}$ defined as $g_t=(1-t)g_0+tg_1$.

 Let $\Xc_0$ and $\Xc_1$ be two representatives of a combing $[X]$. Let $\Xc\colon [0,1]\times \check{M} \to U\check{M}$ be a smooth path
 of
 sections of $U\check{M}$ from $\Xc_0$ to $\Xc_1$ such that $\Xc$ is constant equal to $(0,0,1)$ on $[0,1]\times \check{V}_\infty$.
 
Let $\pi\colon [0,1] \times \check{M} \to \check{M}$ denote the projection on the second factor $\check{M}$. Consider the pull-back bundle $\pi^*(T\check{M})$ over $[0,1]\times \check{M}$. Let $\Xc^{\perp}$ denote the subbundle of $\pi^*(T\check{M})$ whose fiber over a point $(t,m)\in [0,1]\times \check{M}$ is the orthogonal in $T_m\check{M}$ of $\Xc_t(m)$ with respect to the scalar product $g_t(m)$.

 Let $p\in\N\setminus\{0\}$. We define the $S^1$-principal bundle $U\Xc^{\perp}$ over $[0,1]\times \check{M}$ and its quotient $U\Xc^{\perp}_p$ by $\U_p$ as before. They are equipped with a preferred trivialisation over $[0,1]\times \check{V}_\infty$.

Assume that $p$ is a multiple of the order of $e(U\Xc_0^\perp)=e(U\Xc_1^\perp)$. For $j\in\{0,1\}$, let $\nu_{\Xc_j}^p$ be a $p$-multisection of $U(\Xc_j)^{\perp}_p$, where the orthogonal is taken with respect to the metric $g_j$. The proof of Lemma~\ref{lem_exist_section} can be easily adapted
 to show the existence of a section of $U\Xc^{\perp}_p$ that is constant equal to $(1,0,0)$ on $[0,1] \times \check{V}_\infty$ and that restricts to $\nu_{\Xc_j}^p$ on $\{j\}\times\check{M}$ for $j\in\{0,1\}$. Let $\nu^p_\Xc$ denote any such section.

\subsection{Forms associated with paths of parameters}

Let $g,\Xc,p,\nu_\Xc^p$ be as in the above section. 

Forms on $[0,1]\times S^2$ are called $\U_p$-invariant if they are invariant with respect to the action of $\U_p$ given by $\theta \mapsto \mbox{Id}_{[0,1]} \times R_\theta$. 

We define a map $(\mbox{Id}_{[0,1]} \times p_{(\Xc,\nu_\Xc^p)})^*$ from the space of $\U_p$-invariant forms on $[0,1]\times S^2$ to the space of forms on $[0,1]\times \partial C_2(M)$ as before. Let $B$ be an embedded $3$-ball in $\check{M}$. Let $\nu_\Xc$ be a lift of $\nu_\Xc^p$ over $[0,1]\times B$. Complete $(\Xc,\nu_\Xc)$ into a smooth path $(\Xc,\nu_\Xc,w_\Xc)$ of orthormal basis of $\pi^*(TM|_{B})$ with respect to $g$. We get a projection $p_{(\nu_\Xc,w_\Xc,\Xc)}\colon [0,1] \times U\check{M}|_{B} \to [0,1]\times S^2$.  When $\omega$ is an $\U_p$-invariant form on
 $[0,1]\times S^2$, we set
$$(\mbox{Id}_{[0,1]} \times p_{(\Xc,\nu_\Xc^p)})^*(\omega)=(\mbox{Id}_{[0,1]} \times p_{(\nu_\Xc,w_\Xc,\Xc)})^*(\omega)$$
on $[0,1] \times U\check{M}|_{B} \subset [0,1] \times \partial C_2(M)$.

\begin{lema}
 Let $\omega^j$ be a propagating form of $(C_2(M),g_j,\Xc_j,\nu^p_{\Xc_j})$ for $j\in\{0,1\}$. Let $\omega_{S^2}^0,\omega_{S^2}^1$ be two $\U_p$-invariant $2$-forms of volume one of $S^2$ such that
$$\omega^j= p_{(\Xc_j,\nu^p_{\Xc_j})}^*(\omega_{S^2}^j)$$
on $\partial C_2(M)$ for $j\in\{0,1\}$.
Then there exists a closed $\U_p$-invariant $2$-form $\Tilde{\omega}_{S^2}$ on $[0,1]\times S^2$ that restricts to
$\{0\}\times S^2$ as $\omega_{S^2}^0$, and to $\{1\}\times S^2$ as $\omega_{S^2}^1$. 
Let $\omega^{\partial}$ denote the closed $2$-form on $\partial([0,1]\times C_2(M))$ that restricts as
$$\omega^{\partial}=(\mbox{Id}_{[0,1]} \times p_{(\Xc,\nu_\Xc^p)})^*(\Tilde{\omega}_{S^2})$$
 on $[0,1]\times \partial C_2(M)$, and that restricts as $\omega^j$  on $\{j\}\times C_2(M)$ for $j\in\{0,1\}$. There exists a closed $2$-form $\omega$ on $[0,1]\times C_2(M)$ whose restriction to $\partial ([0,1]\times  C_2(M))$ is $\omega^{\partial}$.

\end{lema}
\begin{proof}
  Let $\eta$ be any primitive of the difference $\omega_{S^2}^1-\omega_{S^2}^0$. Define a $1$-form $\Tilde{\eta}$ on $S^2$ as
 $$\Tilde{\eta}=\frac{1}{p} \sum_{k\in\{1,\dots,p\}} R_{\frac{2k\pi}{p}}^*(\eta).$$
 Then $\Tilde{\eta}$ is $\U_p$-invariant. Moreover, we have $d\Tilde{\eta}=\omega_{S^2}^1-\omega_{S^2}^0$. Now set
 $$\Tilde{\omega}_{S^2}=p_2^*(\omega_{S^2}^0)+d(tp_2^*(\Tilde{\eta})),$$
 where $t$ is the coordinate on $[0,1]$ and $p_2\colon [0,1]\times S^2 \to S^2$ is the projection on the second factor. The inclusion $\partial ([0,1]\times  C_2(M)) \subset [0,1]\times  C_2(M)$ induces a surjective map $H^2([0,1]\times  C_2(M);\R)\to H^2(\partial([0,1]\times   C_2(M));\R)$. Indeed, the group $H^3([0,1]\times  C_2(M),\partial([0,1]\times   C_2(M)),\R)$ is isomorphic to $H_4([0,1]\times  C_2(M);\R)$ by Poincaré duality. As $H_*(C_2(M);\R)=H_*(S^2;\R)$, this last group is trivial. This shows the existence of a closed $2$-form $\omega$ with the required properties.
\end{proof}

\subsection{Dependence on the propagating forms}

Let $g,\Xc,p,\nu_\Xc^p$ be as in the above section. For $j\in\{0,1\}$, let $(\omega_i^j)_{i\in\{1,\dots,3n\}}$ be a family of propagating forms of $(C_2(M),g_j,\Xc_j,\nu^p_{\Xc_j})$. For $j\in\{0,1\}$, let $z_n(j)$ denote the element
$z_n(M,(\omega_i^j)_{i\in\{1,\dots,3n\}})$
of $\Ac_n^c(\emptyset)$. In this section, we show $z_n(0)=z_n(1)$.

Let $(\omega_{i,S^2}^j)_{i\in\{1,\dots,3n\}}$ be a family of $\U_p$-invariant $2$-forms of volume one on $S^2$ such that
$$\omega^j_i= p_{(X_j,\nu^p_{X_j})}^*(\omega_{i,S^2}^j)$$
on $\partial C_2(M)$ for all $i\in \{1,\dots,3n\}$ and $j\in\{0,1\}$. Let $(\Tilde{\omega}_{i,S^2})_{i\in\{1,\dots,3n\}}$ be a family of $\U_p$-invariant closed $2$-forms on $[0,1]\times S^2$ and let $(\omega_i)_{i\in\{1,\dots,3n\}}$ be a family of closed $2$-forms on $[0,1]\times C_2(M)$ as in the previous lemma.

 For any $(\Gamma,j)$ in $\Dc_n^c(\emptyset)$, let $(C_{V(\Gamma)}(M))_1$ denote the set of open codimension one faces of $C_{V(\Gamma)}(M)$. Let $o(\Gamma)$ be a vertex-orientation of $\Gamma$. For $F$ in $(C_{V(\Gamma)}(M))_1$, let
$I(F,\Gamma,j)[\Gamma]$
denote the element
$$\Bigl(\int_{[0,1]\times F} \bigwedge_{e\in E(\Gamma)} p_e^*(\omega_{j(e)})\Bigr) [\Gamma,o(\Gamma)]$$ of $\Ac_n^c(\emptyset)$, which does not depend on $o(\Gamma)$. We apply Stokes' theorem to each of the forms
$\bigwedge_{e\in E(\Gamma)} p_e^*(\omega_{j(e)})$
on $[0,1]\times C_{V(\Gamma)}(M)$ to get
$$z_n(1)-z_n(0)= \xi_n \sum_{(\Gamma,j)\in \Dc_n^c(\emptyset)} \sum_{F\in (C_{V(\Gamma)}(M))_1} I(F,\Gamma,j)[\Gamma].$$

\begin{notaa}\label{nota_anomalous}
Let $F(V(\Gamma))$ denote the fiber bundle over $\check{M}$ whose fiber over a point $m\in\check{M}$ is the space $\check{S}_{V(\Gamma)}(T_m\check{M})$ of injections of $V(\Gamma)$ into $T_m\check{M}$ up to translation and dilation. The fiber bundle $F(V(\Gamma))$ is identified with an open codimension one face of $C_{V(\Gamma)}(M)$. It is often called the \emph{anomalous face}. For any edge $e\in E(\Gamma)$, the restriction of $p_e$ to the face  $F(V(\Gamma))$ is valued in $U\check{M}$.
\end{notaa}

The next proposition is a direct corollary of Proposition 9.2 of \cite{lesbook}. We will not give its proof. The more general proposition of Lescop gives in particular a formula for the variation $z_n(1)-z_n(0)$ with the only assumption that the forms $\omega_i$ are of a prescribed form on $[0,1]\times ((\partial C_2(M)) \setminus U\check{M})$. We can apply it in our case, as we have stronger assumptions on our forms $\omega_i$.

\begin{propa}\label{prop_les}
We have
 $$z_n(1)-z_n(0)= \xi_n \sum_{(\Gamma,j)\in \Dc_n^{c}(\emptyset)} I(F(V(\Gamma)),\Gamma,j)[\Gamma].$$
\end{propa}

\begin{lema}\label{lem_anomalous}
 Let $(\Gamma,j)$ be in $\Dc_n^c(\emptyset)$. We have
 $$I(F(V(\Gamma)),\Gamma,j)[\Gamma]=0.$$
 
\end{lema}
\begin{proof}
We will prove that the form $\bigwedge_{e\in E(\Gamma)} (\mbox{Id}_{[0,1]} \times p_e)^*(\omega_{j(e)})$ is $0$ on $[0,1]\times F(V(\Gamma))$.
 Let $B$ be
 an embedded 
 $3$-ball in $\check{M}$ and let $\nu_\Xc$ be a lift of $\nu^p_\Xc$ on $[0,1]\times B$. Complete $(\Xc,\nu_\Xc)$ to an orthonormal frame $(\Xc,\nu_\Xc,w_\Xc)$ of $\pi^*(TM)$ over $[0,1] \times B$. This induces a diffeomorphism $\psi_{(\nu_\Xc,w_\Xc,\Xc)}$ from $F(V(\Gamma))|_{B}$ to $B\times \check{S}_{V(\Gamma)}(\R^3)$. Recall the projection $p_{(\nu_\Xc,w_\Xc,\Xc)}\colon [0,1]\times U\check{M}|_{B} \to S^2$. On $F(V(\Gamma))|_{B}$ we have
 $$\bigwedge_{e\in E(\Gamma)}  (\mbox{Id}_{[0,1]} \times p_e)^*(\omega_{j(e)}) = \bigwedge_{e\in E(\Gamma)}  (p_{(\nu_\Xc,w_\Xc,\Xc)} \circ (\mbox{Id}_{[0,1]} \times p_e))^*(\Tilde{\omega}_{S^2,j(e)}).$$
  Define $p_e'\colon \check{S}_{V(\Gamma)}(\R^3) \to S^2$ to be $G_{S^3}\circ p_e$. On $F(V(\Gamma))|_{B}$ we have
 $$p_{(\nu_\Xc,w_\Xc,\Xc)} \circ (\mbox{Id}_{[0,1]} \times p_e) =\mbox{Id}_{[0,1]} \times ( p_e' \circ p_{\check{S}_{V(\Gamma)}(\R^3)} \circ \psi_{(\nu_\Xc,w_\Xc,\Xc)}),$$
 where $p_{\check{S}_{V(\Gamma)}(\R^3)}\colon B\times \check{S}_{V(\Gamma)}(\R^3) \to  \check{S}_{V(\Gamma)}(\R^3)$ is the projection on the factor $\check{S}_{V(\Gamma)}(\R^3)$. Thus on $F(V(\Gamma))|_{B}$ we have
 \begin{multline*}
 \bigwedge_{e\in E(\Gamma)}  (\mbox{Id}_{[0,1]} \times p_e)^*(\omega_{j(e)})\\ = (\mbox{Id}_{[0,1]} \times (p_{\check{S}_{V(\Gamma)}(\R^3)} \circ \psi_{(\nu_\Xc,w_\Xc,\Xc)}))^*\Bigl(\bigwedge_{e\in E(\Gamma)}  (\mbox{Id}_{[0,1]} \times p_e')^*(\omega_{S^2,j(e)})\Bigr).
 \end{multline*}
 The form $\bigwedge_{e\in E(\Gamma)}  (\mbox{Id}_{[0,1]} \times p_e')^*(\omega_{S^2,j(e)})$ is of degree $2|E(\Gamma)|$. It is defined on $[0,1]\times \check{S}_{V(\Gamma)}(\R^3)$, which is of dimension $2|E(\Gamma)|-3$. Thus it vanishes.
\end{proof}

\begin{proof}[Proof of Theorem~\ref{thm_inv_combing}]
Let $[X]$ be a combing on $\check{M}$. Let $k\geq 1$ be an integer. Proposition~\ref{prop_les} and Lemma~\ref{lem_anomalous} show the independence of $z_n$ with respect to the choices of the Riemannian metric $g$, of the particular representative $X$ of $[X]$, of the $k\order(X)$-multisection $\nu^{k\order(X)}_{X}$ of $UX^\perp$, and of the family $(\omega_i)_{i\in\{1,\dots,3n\}}$ of propagating forms of $(C_2(M),g,X,\nu^{k\order(X)}_{X})$. Let $z_n(M,X,k)$ denote this quantity. We prove $z_n(M,X,k)=z_n(M,X,1)$.

Consider the quotient map $\pi_k\colon UX_{\order(X)}^\perp \to UX_{k \order(X)}^\perp$. An $\order(X)$-multisection $\nu_X^{\order(X)}$ of $UX^\perp$ induces a $k\order(X)$-multisection $\pi_k \circ \nu_X^{\order(X)}$ of $UX^\perp$. Any local lift of 
$\nu_X^{\order(X)}$ is a local lift of $\pi_k \circ \nu_X^{\order(X)}$.
This implies that a propagating form of $(C_2(M),g,X, \pi_k \circ \nu_X^{\order(X)})$ is also a propagating form of $(C_2(M),g,X, \nu_X^{\order(X)})$. We obtain $z_n(M,X,k)=z_n(M,X,1)$.
\end{proof}

\section{The variation formula}
\label{sec_proof_var}

 Let $[X]$ be a fixed combing on $\check{M}$. Let $g$ be a fixed Riemannian metric on $\check{M}$. Let $p=\order(X)$. Recall the knot $K$ from Lemma~\ref{lem_exist_knot}. In this section, we prove Theorem~\ref{thm_inv_var}. Our strategy is the following. In \cite[Chapter 19]{lesbook}, Lescop introduced a way of correcting the boundary of propagators using generalised parallelizations called pseudo-parallelizations. She defined an invariant $z(M,\Tilde{\tau})$ of a rational homology sphere $M$ equipped with a pseudo-parallelization $\Tilde{\tau}$. She defined an invariant $p_1$ of pseudo-parallelizations. She proved that $z(M,\Tilde{\tau})=z(M)+1/4 p_1(\Tilde{\tau})\beta$. In Section~\ref{sec_pseudo_par}, we review the definition of pseudo-parallelizations and the above-mentionned results. In Section~\ref{sec_pseudo_par_combing} we construct an appropriate pseudo-parallelization $\Tilde{\tau}$ on $\check{M}$ in the setting described in Section~\ref{sec_concrete_multisection}. In Section~\ref{sec_proof_var_special}, we prove Theorem~\ref{thm_inv_var} when $\slk(K)=0$. In Section~\ref{sec_conclude}, we conclude the proof of Theorem~\ref{thm_inv_var}.

 \subsection{Review of pseudo-parallelizations}
 \label{sec_pseudo_par}
 
 In this section, we introduce pseudo-parallelizations of $\check{M}$ and homogeneous propagating forms of $C_2(M)$ associated with pseudo-parallelizations. In Theorem~\ref{thm_pseudo_par}, we review Lescop's definition of an invariant of rational homology spheres equipped with a pseudo-parallelization. We state the relation between this invariant and $z$. We refer to \cite[Chapter 19]{lesbook} for more general statements and for proofs.
 
  Let $N(\partial [-1,1])=[-1,1]\setminus [-1+\epsilon,1-\epsilon]$. 
  We assume that the map $\kiS{-1}{1}$ of Definition~\ref{def_twist} is constant on each connected component of $N(\partial [-1,1])$.
  Let $N(\partial [a,b])$ be defined in a similar way. The coordinates on $[a,b]\times [-1,1]$ are denoted by $(x,y)$.
  
 \begin{defia}
 A pseudo-parallelization $(N(\gamma);\tau_e,\tau_b)$ of $\check{M}$ consists of
 \begin{itemize}
  \item a framed link $\gamma$ equipped with a neighborhood $N(\gamma)=[a,b]\times \gamma \times [-1,1] \subset M\setminus V_\infty$
  \item a parallelization $\tau_e$ of $\check{M}$ outside $N(\gamma)$
  \item a parallelization $\tau_b$ of $N(\gamma)$ such that $\tau_b=\tau_e$ on $([a,b]\times \gamma \times N(\partial [-1,1])) \cup ([b-\epsilon,b] \times \gamma \times [-1,1])$ and $\tau_b=\tau_e \circ T_\gamma$ on $[a,a+\epsilon]\times \gamma\times [-1,1]$, where $T_\gamma$ is defined on $N(\gamma) \times \R^3$ by
  $$T_\gamma (x,c,y,v)=(x,c,y,R_{2\pi\kiS{-1}{1}(y)}(v)).$$
 \end{itemize}
\end{defia}
In this article, we only consider pseudo-parallelizations $(N(\gamma);\tau_e,\tau_b)$ of $\check{M}$ such that $\tau_e$ is the standard parallelization of $\R^3$ over $V_\infty \setminus \{\infty\}$.

Let $F\colon [a,b]\times [-1,1] \to SO(3)$ be a smooth map such that
$$F(x,y)=
\left\{
\begin{array}{lcl}
 1& \mbox{ if }& |y|>1-\epsilon \\
 R_{2\pi\kiS{-1}{1}(y)} &\mbox{ if }& x<a+\epsilon \\
 R_{2\pi\kiS{-1}{1}(y)}^{-1} &\mbox{ if }& x>b-\epsilon
\end{array}
\right..$$
Such a map $F$ exists because its restriction to the boundary of $[a,b]\times [-1,1]$ is twice an element of $\pi_1(SO(3))=\Z/2\Z$. Later we will use the following lemma. It is equivalent to Lemma 21.13 of $\cite{lesbook}$.

\begin{lema}\label{lem_degree_mone}
 Let $f \colon [a,b]\times [-1,1] \to S^2$ be defined
 by $f(x,y)=F(x,y)((0,0,1))$. Then the degree of the induced map $f$ from $[a,b]\times [-1,1] /\partial ([a,b]\times [-1,1])$ to $S^2$ is $1$.
\end{lema}

\begin{defia} 
Let $\Tilde{\tau}=(N(\gamma);\tau_e,\tau_b)$ be a pseudo-parallelization of $\check{M}$. Let $\omega_{S^2}$ denote the standard $2$-form of volume one on $S^2$. Let $F_\gamma\colon N(\gamma) \times \R^3 \to N(\gamma) \times \R^3$ be defined
by
$$F_\gamma (x,c,y,v)=(x,c,y,F(x,y)(v)).$$
We define a closed $2$-form $\omega(\gamma,\tau_b)$ on $N(\gamma)$ as
 $$\omega(\gamma,\tau_b)=\frac{p_{\tau_b \circ T_{\gamma}^{-1}}^*(\omega_{S^2}) + p_{\tau_b \circ F_{\gamma}^{-1}}^*(\omega_{S^2})}{2}$$
 The \emph{homogeneous boundary form} associated to $(\Tilde{\tau},F)$ is the closed $2$-form $\omega(\Tilde{\tau},F)$ defined on $U\check{M}$ as
 $$\omega(\Tilde{\tau},F)=
 \left\{
 \begin{array}{lcl}
  p_{\tau_e}^*(\omega_{S^2}) &\mbox{ on }& U\check{M}|_{ \check{M} \setminus N(\gamma)} \\
  \omega(\gamma,\tau_b) &\mbox{ on }& U\check{M}|_{ N(\gamma)}
 \end{array}
\right..
 $$
 \end{defia}
 
 \begin{defia}
 Let $\Tilde{\tau}$ be a pseudo-parallelization of $\check{M}$. A \emph{homogeneous propagating form} of $(C_2(M),\Tilde{\tau})$ is a propagating form of $C_2(M)$ that restricts as the homogeneous boundary form associated to $(\Tilde{\tau},F)$  on $U\check{M}$ for some map $F$ as above.
 \end{defia}
 
 \begin{thma}\label{thm_pseudo_par}
  Let $\Tilde{\tau}$ be a pseudo-parallelization of $\check{M}$. Let $\omega$ be a homogeneous propagating form of $(C_2(M),\Tilde{\tau})$. The element
  $$z_n(M,\Tilde{\tau},\omega)=\xi_n \sum_{(\Gamma,j)\in \Dc_n^c(\emptyset)} \int_{C_{V(\Gamma)}(M)} \bigwedge_{e\in E(\Gamma)} p_e^*(\omega) [\Gamma]$$
  of $\A_n^c(\emptyset)$ does not depend on the choice of $\omega$. It is denoted by $z_n(M,\Tilde{\tau})$. There exists a $\Q$-valued map $p_1$ on the 
  set of pseudo-parallelizations of $\check{M}$ such that
  $$z_n(M,\Tilde{\tau})=z_n(M)+\frac{1}{4}p_1(\Tilde{\tau})\beta_n.$$
 \end{thma}

 Lescop used the more flexible definition of $Z$ with pseudo-parallelizations to prove the universality of $Z$ among finite type invariants of rational homology spheres. See \cite[Part IV]{lesbook}.

 \subsection{A pseudo-parallelization associated with a combing}
 \label{sec_pseudo_par_combing}
 
 Recall our construction from Section~\ref{sec_concrete_multisection}. In particular, recall the notation $K,N(K),E_K$, the section $Y_e$ of $X^\perp$ over $E_K$, and the parallelization $\tau_e$ of $E_K$ from Lemma~\ref{lem_exist_knot}. Also recall that the third vector of $\tau_e$ is $X$.

In this section, we define a parallelization $\tau_N$ of $N(K)$ such that
the third vector of $\tau_N$ is $X$, and $\Tilde{\tau}=(N(K);\tau_e,\tau_{N})$ is a pseudo-parallelization. Thus, we define a pseudo-parallelization $\Tilde{\tau}$ of $\check{M}$ that extends $\tau_e$. We distinguish two cases.
 
When $\slk(K)\neq 0$, we use the following lemma as a definition for $\tau_N$.
 \begin{lema} 
 There exists a parallelization $\tau_{N}$ of $N(K)$ such that $(N(K);\tau_e,\tau_{N})$ is a pseudo-parallelization of $\check{M}$ and $X$ is the third vector of $\tau_N$.
 \end{lema}
 \begin{proof}
  Consider the section $Y$ of $UX^\perp$ defined over $E_K\cap N(K)$ as 
  $$Y=
  \left\{
  \begin{array}{lcl}
   Y_e & \mbox{ if } & x\in [b-\epsilon,b] \mbox{ or } y\in N(\partial [-1,1])) \\
   R_{2\pi\kiS{-1}{1}(y)}\circ Y_e & \mbox{ if } & x\in [a,a+\epsilon]
  \end{array}
\right..
  $$
  Let $s$ be a section of $UX^\perp$ over $D_m$. We have $d_{m_K}(s,Y_e)=-1$ and $d_{m_K}(Y_e,Y)=1$. Lemma~\ref{lem_deg_sum} then implies $d_{m_K}(s,Y)=0$. Lemma~\ref{lem_deg_ext} shows that $Y$ extends as a section of $UX^\perp$ over $N(K)$. Complete $(X,Y)$ to an orthonormal frame $(X,Y,W)$ of $UN(K)$. Let $\tau_N$ be the parallelization of $N(K)$ associated to $(Y,W,X)$. One easily check that $(N(K);\tau_e,\tau_{N})$ is a pseudo-parallelization of $\check{M}$.
 \end{proof}

 When $\slk(K)= 0$, we use the following lemma as a definition for $\tau_N$, with the parallelization $\tau_N^{\mbox{\scriptsize multi}}$ of $N(K)$ defined in Lemma~\ref{lem_para_int}.

 \begin{lema} 
Assume that $\slk(K)=0$. There exists a map $g$ from $[a,b]\times [-1,1] \to SO(2)\subset SO(3)$ of rotations of $\R^3$ around $(0,0,1)$ such that the triplet $(N(K);\tau_e,\tau_N^{\mbox{\scriptsize multi}} \circ g_K)$ is a pseudo-parallelization of $\check{M}$, where $g_K$ is defined on $N(K)$ as $g_K(x,c,y)=g(x,y)$. Let $\tau_N=\tau_N^{\mbox{\scriptsize multi}} \circ g_K$.
\end{lema}
\begin{proof}
As a consequence of our definition of $\tau_N^{\mbox{\scriptsize multi}}$ in Lemma~\ref{lem_para_int}, there exists a map $R \colon [-1,1] \to SO(2)\subset SO(3)$ of rotations around $(0,0,1)$ such that
$$\tau_N^{\mbox{\scriptsize multi}} (x,c,y,v) = \tau_e (x,c,y,R(y)(v))$$
for $(x,c,y,v)$ in $[a,a+\epsilon ]\times S^1 \times [-1,1] \times \R^3$. Moreover, the map $R$ is homotopic relatively to $N(\partial [-1,1])$ to $y\mapsto R_{2\pi\kiS{-1}{1}(y)}$.
We choose a map $g$ from $[a,b]\times [-1,1]$ to $SO(2)\subset SO(3)$ such that
$$g=
\left\{
\begin{array}{lcl}
 R^{-1} \circ R_{2\pi\kiS{-1}{1}(y)} & \mbox{ on } & [a,a+\epsilon] \times [-1,1] \\
 1 & \mbox{ on } & [a+2\epsilon,b] \times [-1,1] \\
 1  & \mbox{ on } & [a,b] \times N(\partial [-1,1])
 \end{array}
\right..
$$
The result follows.
\end{proof}

\subsection{Proof of the variation formula in a special case}
\label{sec_proof_var_special}

Recall the $p$-multisection $\Tilde{Y}^p$ of $UX^\perp$ from Lemma~\ref{lem_concrete_multi}. Recall that $u\in[-1,1]$ denotes the normal coordinate to $\Sigma$. In the following lemma, we compare the homogeneous boundary form $\omega(\Tilde{\tau},F)$ and $p_{(X,\Tilde{Y}^p)}^*(\omega_{S^2})$ on $U E_K$. It is a direct corollary of \cite[Lemma 19.14]{lesbook}.

\begin{lema}\label{lem_compare_forms}
Let $v_3$ denote the third coordinate on $S^2\subset \R^3$. On $U(E_K \setminus N(\Sigma))$ we have
 $$p_{(X,\Tilde{Y}^p)}^*(\omega_{S^2}) =\omega(\Tilde{\tau},F).$$
 On $U N(\Sigma)$ we have
 $$p_{(X,\Tilde{Y}^p)}^*(\omega_{S^2}) = \omega(\Tilde{\tau},F) + \frac{1}{2p} d(v_3\circ p_{\tau_e}) \wedge d\kiS{-1}{1}(u).$$
\end{lema}

In the next lemma, which will be proved after Lemma~\ref{lem_homo}, we construct a closed $2$-form on $[0,1]\times \partial C_2(M)$ which interpolates between $\omega(\Tilde{\tau},F)$ and $p_{(X,\Tilde{Y}^p)}^*(\omega_{S^2})$. The
$[0,1]$-coordinate
is denoted by $t$. We choose a smooth function $\kiSneg{0}{1}\colon [0,1]\to [0,1]$ that is $1$ on a neighbourhood of $0$ and $0$ on a neighbourhood of $1$.

\begin{lema}\label{lem_forme_bord}
Let $p_2\colon [0,1]\times \partial C_2(M) \to  \partial C_2(M)$ denote the projection on the second factor. Let $\Omega$ be the closed $2$-form defined on a subset of $[0,1]\times \partial C_2(M)$ as
$$
\Omega=
\left\{
\begin{array}{lcl}
p_{(X,\Tilde{Y}^p)}^*(\omega_{S^2}) & \mbox{ on }& \{0\} \times \partial C_2(M)\\
(p_{\tau_e}\circ p_2)^*(\omega_{S^2})& \mbox{ on }& [0,1] \times U(E_K \setminus  N(\Sigma))\\
(p_{\tau_e}\circ p_2)^*(\omega_{S^2})+ \frac{1}{2p} d\bigl(\kiSneg{0}{1}(t) v_3\circ p_{\tau_e}\circ p_2 d(\kiS{-1}{1}\circ u\circ p_2)\bigr) & \mbox{ on }& [0,1] \times UN(\Sigma)\\
(G_M \circ p_2)^*(\omega_{S^2}) & \mbox{ on }& [0,1] \times ((\partial C_2(M)) \setminus U\check{M}) \\
\omega(\Tilde{\tau},F) & \mbox{ on }& \{1\} \times \partial C_2(M)\\
\end{array}
\right..
$$
There exists a closed $2$-form $\Omega((X,\Tilde{Y}^p),\Tilde{\tau})$ on $[0,1]\times \partial C_2(M)$ which extends $\Omega$.

When $\slk(K)=0$, there exists such a form $\Omega((X,\Tilde{Y}^p),\Tilde{\tau})$ with the following additional property. There exists a closed $2$-form $\Omega((X,\Tilde{Y}^p),\Tilde{\tau})'$ on $[a,b]\times [-1,1]\times S^2$ such that $\Omega((X,\Tilde{Y}^p),\Tilde{\tau})$ is the pull-back of $\Omega((X,\Tilde{Y}^p),\Tilde{\tau})'$ through the identification 
of $UN(K)$ with
$[a,b]\times S^1 \times [-1,1]\times S^2$ given by $\tau_{N}$ and the projection on the factor $[a,b] \times [-1,1]\times S^2$.
\end{lema}

To prove Lemma~\ref{lem_forme_bord}, we first note that
 Lemma~\ref{lem_compare_forms} implies that the form $\Omega$ is well-defined. It is in particular well-defined on
 $$\mathcal{U}=\bigl(\{0,1\}\times UN(K)\bigr) \cup \bigl([0,1] \times U(E_K\cap N(K))\bigr) \subset [0,1]\times UN(K),$$ which deformation retracts onto $\partial ([0,1]\times UN(K))$.  We would like to extend $\Omega$ on the remaining part of $[0,1]\times UN(K)$. The next easy lemma
 determines the homological obstruction to such an extension.

 \begin{lema}\label{lem_homo}
 Recall that $D_m$ denotes the meridian disk $[a,b]\times \{1\} \times [-1,1]$ of $N(K)$, oriented as $-[a,b]\times [-1,1]$. Let $\Gamma(X|_{D_m})\subset UN(K)$ denote the graph of the restriction of $X$ to $D_m$, oriented as $D_m$. Let $S_X$ denote the $2$-dimensional cycle $\partial([0,1]\times \Gamma(X|_{D_m}))$ of $\partial([0,1]\times UN(K))$.
 
 The image of the map $H^2_{dR}([0,1]\times UN(K);\R) \to H^2_{dR}(\mathcal{U};\R)$ induced by the inclusion is the space of homology classes of closed $2$-forms $\omega$ on $[0,1]\times UN(K)$ such that
 $\int_{S_X} \omega =0$.
 \end{lema}
 
 \begin{proof}[Proof of Lemma~\ref{lem_forme_bord}]
 We will prove that $\int_{S_X} \Omega =0$. We first compute the integral of $\Omega$ over $(-[0,1]\times \Gamma(X|_{\partial D_m}))\subset S_X$.
 The section $X$ is the third vector of $\tau_e$, so $v_3 \circ p_{\tau_e} \circ p_2$ is constant equal to $1$ on $[0,1]\times \Gamma(X|_{\partial D_m})$. Only the graph of $X$ over the $p$ segments of $\partial D_m \cap N(\Sigma)$ contribute to the integral, and each contributes in the same way. Recall that $\partial D_m$ is oriented as the positive normal to $\Sigma$. We get
 \begin{align*}
 \int_{-[0,1]\times \Gamma(X|_{\partial D_m})} \Omega  & = p\times \frac{1}{2p} \times  \int_{- [0,1]\times [-1,1]} d\bigl(\kiSneg{0}{1}(t)  d\kiS{-1}{1}(u)\bigr) \\ & = \frac{1}{2},
 \end{align*}
 where the last equality is a straighforward application of Stokes' theorem.
 
 The form $p_{(X,\Tilde{Y}^p)}^*(\omega_{S^2})$ is locally defined as the pull-back of $\omega_{S^2}$ by a map to $S^2$ associated with a local parallelization having $X$ as its third vector. So it is identically $0$ on the graph of $X$ over $\check{M}$. This implies that the integral of $\Omega$ over $- \{0\} \times \Gamma(X|_{D_m})$ is $0$.
 
 We prove that the integral of $\Omega$ over $\{1\} \times \Gamma(X|_{D_m})$ is $-1/2$. The restriction of $p^*_{\tau_{N} \circ T_K^{-1}}(\omega_{S^2})$ to the graph of $X$ over $N(K)$ is $0$, because $X$ is the third vector of $\tau_{N} \circ T_K^{-1}$. Recall the map $f$ from Lemma~\ref{lem_degree_mone}. Then $f$ is exactly the restriction of $p_{\tau_{N} \circ F_K^{-1}}$ to $\Gamma(X|_{D_m})$. Indeed, for $(x,y)\in [a,b]\times [-1,1]$ we have
 \begin{align*}
 \tau_N \circ F_K^{-1} (x,1,y,f(x,y))&=\tau_N \circ F_K^{-1} (x,1,y,F(x,y)(0,0,1))\\&=\tau_N(x,1,y,(0,0,1))\\&=X(x,1,y).
 \end{align*}
 We get
 \begin{align*}
 \int_{\{1\} \times \Gamma(X|_{D_m})} \Omega  & = \frac{1}{2} \int_{- [a,b]\times [-1,1]} f^*(\omega_{S^2}) \\ & = -\frac{1}{2}.
 \end{align*}
 This concludes the proof of the first part of the lemma.
 
 Assume now that $\slk(K)=0$. Consider the identification between $UN(K)$ and $[a,b]\times S^1 \times [-1,1]\times S^2$ given by $\tau_{N}$. Let $\mathcal{U}'$ denote the image of $\mathcal{U}$ by the projection of $[0,1]\times UN(K)$ onto the factor $[0,1]\times [a,b]\times [-1,1]\times S^2$. The restriction of the form $\Omega$ to $\mathcal{U}$ is the pull-back of a $2$-form $\Omega'$ on $\mathcal{U}' $. This form $\Omega'$ extends on $[0,1]\times [a,b]\times [-1,1] \times S^2$ as before. We define $\Omega((X,\Tilde{Y}^p),\Tilde{\tau})'$ on $[0,1]\times UN(K)$ to be the pull-back of $\Omega'$ by the above-mentionned projection.
\end{proof}

In the next proposition, we use the form $\Omega$ of Lemma~\ref{lem_forme_bord} to compare $z_n(M,X)$ and $z_n(M,\Tilde{\tau})$. When $\Gamma$ is a trivalent graph, recall the description of the anomalous face $F(V(\Gamma))$ of $C_{V(\Gamma)}(M)$ from Notation~\ref{nota_anomalous}.

\begin{propa}\label{prop_pseudo_combing}
  Let $\Omega((X,\Tilde{Y}^p),\Tilde{\tau})$ be as in Lemma~\ref{lem_forme_bord}. We have
  $$z_n(M,\Tilde{\tau})=z_n(M,X) + \sum_{(\Gamma,j)\in \Dc_n^e(\emptyset)} \int_{[0,1] \times F(V(\Gamma))|_{N(K)}} \bigwedge_{e\in E(\Gamma)} p_e^*(\Omega((X,\Tilde{Y}^p),\Tilde{\tau})).$$
  When $\slk(K)=0$, we obtain
  $$z_n(M,\Tilde{\tau})=z_n(M,X).$$
\end{propa}
\begin{proof}
 Let $\omega_0$ be a propagating form of $(C_2(M),X,\Tilde{Y}^p)$ such that $\omega_0=p_{X,\Tilde{Y}^p}^*(\omega_{S^2})$ on $\partial C_2(M)$. Let $\omega_1$ be a homogeneous propagating form of $(C_2(M),\Tilde{\tau})$. Let $\omega^{\partial}$ be the closed $2$-form on $\partial ([0,1]\times C_2(M))$ such that $(\omega^\partial)|_{\{i\}\times C_2(M)}=\omega_i$ for $i\in\{0,1\}$ and $(\omega^\partial)|_{[0,1]\times \partial C_2(M)}=\Omega((X,\Tilde{Y}^p),\Tilde{\tau})$. Homological arguments again show that there exists a closed $2$-form $\omega$ on $[0,1]\times C_2(M)$ which restricts as $\omega^{\partial}$ on the boundary of $[0,1]\times C_2(M)$.
 
 As in the proof of Theorem~\ref{thm_inv_combing},
\cite[Proposition 9.2]{lesbook} applies 
 in our case, with a constant family of forms $(\omega_i=\omega)_{i\in\{1,\dots,3n\}}$ on $[0,1]\times C_2(M)$. We obtain
 $$z_n(M,\Tilde{\tau})-z_n(M,X)= \sum_{(\Gamma,j)\in \Dc_n^c(\emptyset)} \int_{[0,1] \times F(V(\Gamma))} \bigwedge_{e\in E(\Gamma)} p_e^*(\Omega((X,\Tilde{Y}^p),\Tilde{\tau})).$$
 The parallelization $\tau_e$ induces an identification between  $F(V(\Gamma))|_{E_K}$ and $E_K \times \check{S}_{V(\Gamma)}(\R^3)$.
 
 Let $\Gamma$ be a trivalent graph. The restriction of $\bigwedge_{e\in E(\Gamma)} p_e^*(\Omega((X,\Tilde{Y}^p),\Tilde{\tau}))$ to $[0,1] \times (E_K \setminus N(\Sigma)) \times\check{S}_{V(\Gamma)}(\R^3) $ factors through the projection
 on the factor $\check{S}_{V(\Gamma)}(\R^3)$. The dimension of $\check{S}_{V(\Gamma)}(\R^3)$ is $2|E(\Gamma)|-4$, while the degree of this form is $2|E(\Gamma)|$. Thus, this form vanishes. 
 
 The restriction of the form $\bigwedge_{e\in E(\Gamma)} p_e^*(\Omega((X,\Tilde{Y}^p),\Tilde{\tau}))$ to $[0,1] \times  N(\Sigma) \times \check{S}_{V(\Gamma)}(\R^3)$ factors through the projection
 on the factor $[0,1]\times [-1,1]\times \check{S}_{V(\Gamma)}(\R^3)$. The dimension of this space is $2|E(\Gamma)|-2$. Thus, this form also vanishes. This concludes the proof of the first part of the proposition.
 
 Assume that $\slk(K)=0$. The parallelization $\tau_{N}$ induces an identification between $F(V(\Gamma))|_{N(K)}$ and $[a,b]\times S^1 \times [-1,1] \times \check{S}_{V(\Gamma)}(\R^3)$. The restriction of the form $\bigwedge_{e\in E(\Gamma)} p_e^*(\Omega((X,\Tilde{Y}^p),\Tilde{\tau}))$ to $[0,1] \times N(K)\times\check{S}_{V(\Gamma)}(\R^3) $ factors through the projection
 on the factor
 $[0,1]\times [a,b]\times [-1,1]\times \check{S}_{V(\Gamma)}(\R^3)$. 
 The dimension of this space is $2|E(\Gamma)|-1$. So this form vanishes. The result follows.
 
\end{proof}

\begin{coroa}\label{coro_slknul}
Recall the combing $[X]$ and the knot $K$ from the beginning of Section~\ref{sec_proof_var}. Recall the pseudo-parallelization $\Tilde{\tau}$ of $\check{M}$ from the end of Section~\ref{sec_pseudo_par_combing}. Assume that $\slk(K)=0$. Then $p_1(X)=p_1(\Tilde{\tau})$. Moreover, we have
  $$z_n(M,X)=z_n(M)+\frac{1}{4}p_1(X)\beta_n.$$
\end{coroa}
\begin{proof}
 Apply Proposition~\ref{prop_pseudo_combing} in degree $1$ to get $z_1(M,X)=z_1(M,\Tilde{\tau}).$ Recall from Proposition~\ref{prop_eq_theta_inv} that
 $ z_1(M,X)=z_1(M)+\frac{1}{4}p_1(X)\beta_1$. Theorem~\ref{thm_pseudo_par} in degree $1$ shows  $z_1(M,\Tilde{\tau})=z_1(M)+\frac{1}{4}p_1(\Tilde{\tau})\beta_1$. The equality $p_1(X)=p_1(\Tilde{\tau})$ follows.
 
 Proposition~\ref{prop_pseudo_combing} in degree $n$ implies $z_n(M,X)=z_n(M,\Tilde{\tau})$. Theorem~\ref{thm_pseudo_par} in degree $n$ implies
 $z_n(M,\Tilde{\tau})=z_n(M)+\frac{1}{4}p_1(\Tilde{\tau})\beta_n$. The result follows.
\end{proof}

\subsection{Concluding the proof of the variation formula}
\label{sec_conclude}

In this section, we prove Theorem~\ref{thm_inv_var} when $\slk(K)\neq 0$.

Let $r\in]0,1/4p[$. Let $(B_i(3r))_{i\in\{1,\dots,p\}}$ be a family of pairwise disjoint open balls of radius $3r$ in $B(1)\subset \R^3$.  Recall that a neighbourhood $V_\infty$ of $\infty\in M$ is identified with $S^3 \setminus B(1)$. Let $B_M$ denote the manifold $(M\setminus V_\infty) \cup ( \overline{B(2)}\setminus B(1))$. Let $\connect^p M$ (resp. $\connect^p \check{M}$) denote the manifold obtained by gluing $S^3 \setminus \bigcup_{i\in\{1,\dots,p\}} B_i(r)$ (resp. $\R^3 \setminus \bigcup_{i\in\{1,\dots,p\}} B_i(r)$) and $p$ copies of $B_M$ along $\bigcup_{i\in\{1,\dots,p\}} (\overline{B_i(2r)} \setminus B_i(r))$ and collar neighbourhoods $ \overline{B(2)}\setminus B(1)$ of the boundary of the $p$ copies of $B_M$.

The combings $[X]$ defined on the copies of $B_M$ extend to a combing $[\connect^p X]$ on $\connect^p \check{M}$. The $p$-multisections $\Tilde{Y}^p$ of $UX^\perp$ over $B_M$ extend as a $p$-multisection $\connect^p \Tilde{Y}^p$ of $U(\connect^p X)^\perp$. The pseudo-parallelizations $\Tilde{\tau}$ defined on the copies of $B_M$ extend to a pseudo-parallelization $\connect^p \Tilde{\tau}$ of $\connect^p \check{M}$.

\begin{lema}\label{lem_connect}
We have
 $$z_n(M, \Tilde{\tau}) - z_n(M, X) =\frac{1}{p} \bigl( z_n(\connect^p M,\connect^p \Tilde{\tau}) - z_n(\connect^p M,\connect^p X) \bigr).$$
\end{lema}
\begin{proof}
Let $\Omega((X,\Tilde{Y}^p),\Tilde{\tau})$ be as in Lemma~\ref{lem_forme_bord}. We define a closed $2$-form $\Omega$ on $[0,1]\times \partial C_2(\connect^p M)$.
For each of the $p$ copies $(B_M)_i,i\in\{1,\dots,p\}$ of $B_M$ inside $\connect^p M$, we set
$$\Omega=\Omega((X,\Tilde{Y}^p),\Tilde{\tau})$$
on $[0,1]\times U(B_M)_i$. On $[0,1]\times
U(\R^3 \setminus \bigcup_{i\in\{1,\dots,p\}} B_i(r)) \cup ((\partial C_2(\connect^p M)) \setminus U(\connect^p \check{M}))$, we set
$$\Omega=G_{S^3}^*(\omega_{S^2}).$$

The form $\Omega$ restricts as $p_{(\connect^p X,\connect^p \Tilde{Y}^p)}^*(\omega_{S^2})$ on $\{0\}\times \partial C_2(\connect^p M)$. It restricts as the homogeneous boundary form of $(C_2(\connect^p M),\connect^p \Tilde{\tau})$ on $\{1\}\times \partial C_2(\connect^p M)$. When $\Gamma$ is a trivalent graph, let $F(M,V(\Gamma))$ denote the anomalous face of $C_{V(\Gamma)}(M)$ and let $F(\connect^p M,V(\Gamma))$ denote the anomalous face of $C_{V(\Gamma)}(\connect^p M)$. Using arguments already given in the proof of Proposition~\ref{prop_pseudo_combing}, we get
\begin{align*}
z_n(\connect^p M,\connect^p \Tilde{\tau}) - z_n(\connect^p M,\connect^p X)& =  \xi_n \sum_{(\Gamma,j)\in \Dc_n^c(\emptyset)} \int_{[0,1] \times F(\connect^p M,V(\Gamma))} \bigwedge_{e\in E(\Gamma)} p_e^*(\Omega).
\end{align*}
For any $(\Gamma,j)\in \Dc_n^c(\emptyset)$, the form $\bigwedge_{e\in E(\Gamma)} p_e^*(\Omega)$ is zero on the part of $[0,1] \times F(\connect^p M,V(\Gamma))$ over $[0,1] \times U(\R^3 \setminus \bigcup_{i\in\{1,\dots,p\}} B_i(r))$. Our definition of $\Omega$ implies that the right-hand side of the above equation is
\begin{align*}
 p \times \xi_n \sum_{(\Gamma,j)\in \Dc_n^c(\emptyset)} \int_{[0,1] \times F(M,V(\Gamma))} \bigwedge_{e\in E(\Gamma)} p_e^*(\Omega((X,\Tilde{Y}^p),\Tilde{\tau})),
\end{align*}
which is equal to $p (z_n(M, \Tilde{\tau}) - z_n(M, X))$. The result follows.
\end{proof}

 Applying the above equality for $n=1$ and using the formulas of Proposition~\ref{prop_eq_theta_inv} and Theorem~\ref{thm_pseudo_par}, we get the following corollary of Lemma~\ref{lem_connect}.
\begin{coroa}\label{coro_connect}
We have
 $$p_1(\connect^p \Tilde{\tau})-p_1(\connect^p X)= p\bigl( p_1(\Tilde{\tau})-p_1(X)\bigr).$$
 \end{coroa}

 The group 
 $H^2(\connect^p M;\Z)$ is isomorphic to the direct sum of $p$ copies of the group $H^2(M;\Z)$.  The Euler class of $U(\connect^p X)^\perp$ is $(e(UX^\perp),\dots,e(UX^\perp))$ with respect to this identification. The
 self-linking
 number of the Poincaré dual to $e((\connect^p X)^\perp)$ is the class of $p\times (-q/p)$ in $\Q/\Z$. Thus it is $0$.
 
 \begin{lema}\label{lem_connect_var}
  We have
  $$z_n(\connect^p M,\connect^p X)- z_n(\connect^p M,\connect^p \Tilde{\tau}) = \frac{1}{4}\bigl(p_1(\connect^p X) - p_1(\connect^p \Tilde{\tau})\bigr)\beta_n.$$
 \end{lema}
 \begin{proof}
   Corollary \ref{coro_slknul} yields
  $z_n(\connect^p M,\connect^p X)=z_N(\connect^p M)+\frac{1}{4}p_1(\connect^p X)\beta_n$. Theorem~\ref{thm_pseudo_par} implies that $z_n(\connect^p M,\connect^p \Tilde{\tau})=z_N(\connect^p M)+\frac{1}{4}p_1(\connect^p \Tilde{\tau})\beta_n$. The result follows.
 \end{proof}

\begin{proof}[Proof of Theorem~\ref{thm_inv_var}]
Injecting the equality of Corollary \ref{coro_connect} into Lemma~\ref{lem_connect_var}, we obtain
$$z_n(\connect^p M,\connect^p X)- z_n(\connect^p M,\connect^p \Tilde{\tau}) = \frac{p}{4}\bigl(p_1( X) - p_1(\Tilde{\tau})\bigr)\beta_n.$$
Combining this equality and Lemma~\ref{lem_connect}, we get
$$z_n(M, X) - z_n(M, \Tilde{\tau}) = \frac{1}{4}\bigl(p_1( X) - p_1(\Tilde{\tau})\bigr)\beta_n.$$
As $z_n(M, \Tilde{\tau})=z_n(M)+1/4 p_1(\Tilde{\tau})\beta_n$, the result follows.
\end{proof}

\bibliographystyle{plain}
\bibliography{bibli_comb.bib}

\end{document}